\documentclass{amsart}

\usepackage[
    babel={
        ngerman,main=UKenglish,shorthands=off
    },
]{preamble}

\usepackage{subcaption}

\hypersetup{%
    pdfauthor={Jannis Weis},%
    pdftitle={Polynomial filling functions of sublevel sets in CAT(0) cube complexes},%
    linktoc=all,
    hidelinks,
}
\makeTodoCommand[red]{comjw}{Jannis}

\colorlet{Red}{T-Q-V5}%
\colorlet{light Red}{T-Q-L7}%
\colorlet{Blue}{OI5}%
\colorlet{light Blue}{T-Q-L1}%
\colorlet{Green}{OI3}%
\colorlet{Yellow}{T-Q-L6}%

\usetikzlibrary{perspective,calc,arrows.meta,decorations.markings}

\makeatletter
\tikzset{
    nomorepostaction/.code={
        \let\tikz@postactions\pgfutil@empty
    },
    descending edge/.style={
        postaction={
            nomorepostaction,
            decorate,
            decoration={
                markings,
                mark=at position #1 with {
                    \arrow{Stealth[length=4pt,width=3.5pt]}
                }
            }
        }
    },
    descending edge/.default=0.5
}
\makeatother

\pgfkeys{
    /cube coordinates/.is family,
    /cube coordinates,
    prefix/.store in=\cubeCoordinatePrefix,
    x/.store in=\cubeCoordinateXOffset,
    y/.store in=\cubeCoordinateYOffset,
    z/.store in=\cubeCoordinateZOffset,
}

\NewDocumentCommand{\cubeCoordinates}{O{}}{%
    \pgfkeys{/cube coordinates/.cd,prefix={},x=0,y=0,z=0,#1}%
    \path (tpp cs:x={\cubeCoordinateXOffset},y={\cubeCoordinateYOffset},z={\cubeCoordinateZOffset}) coordinate (\cubeCoordinatePrefix O); %
    \path (tpp cs:x={\cubeCoordinateXOffset},y={\cubeCoordinateYOffset},z={\cubeCoordinateZOffset+\LZ}) coordinate (\cubeCoordinatePrefix BC); %
    \path (tpp cs:x={\cubeCoordinateXOffset},y={\cubeCoordinateYOffset+\LY},z={\cubeCoordinateZOffset}) coordinate (\cubeCoordinatePrefix AC); %
    \path (tpp cs:x={\cubeCoordinateXOffset+\LX},y={\cubeCoordinateYOffset},z={\cubeCoordinateZOffset}) coordinate (\cubeCoordinatePrefix AB); %
    \path (tpp cs:x={\cubeCoordinateXOffset},y={\cubeCoordinateYOffset+\LY},z={\cubeCoordinateZOffset+\LZ}) coordinate (\cubeCoordinatePrefix C); %
    \path (tpp cs:x={\cubeCoordinateXOffset+\LX},y={\cubeCoordinateYOffset},z={\cubeCoordinateZOffset+\LZ}) coordinate (\cubeCoordinatePrefix B); %
    \path (tpp cs:x={\cubeCoordinateXOffset+\LX},y={\cubeCoordinateYOffset+\LY},z={\cubeCoordinateZOffset}) coordinate (\cubeCoordinatePrefix A); %
    \path (tpp cs:x={\cubeCoordinateXOffset+\LX},y={\cubeCoordinateYOffset+\LY},z={\cubeCoordinateZOffset+\LZ}) coordinate (\cubeCoordinatePrefix V); %
}

\pgfkeys{
    /draw cube/.is family,
    /draw cube,
    prefix/.store in=\drawCubePrefix,
}

\NewDocumentCommand{\drawCube}{O{}}{%
    \pgfkeys{/draw cube/.cd,prefix={},#1}%
    \begin{scope}[thick,solid,every path/.append style={descending edge}]
        \draw (\drawCubePrefix BC)--(\drawCubePrefix O);
        \draw (\drawCubePrefix B)--(\drawCubePrefix BC);
        \draw (\drawCubePrefix C)--(\drawCubePrefix BC);
        \draw (\drawCubePrefix AB)--(\drawCubePrefix O);
        \draw (\drawCubePrefix A)--(\drawCubePrefix AB);
        \draw (\drawCubePrefix A)--(\drawCubePrefix AC);
        \draw (\drawCubePrefix AC)--(\drawCubePrefix O);
        \draw (\drawCubePrefix B)--(\drawCubePrefix AB);
        \draw (\drawCubePrefix C)--(\drawCubePrefix AC);
        \draw (\drawCubePrefix V)--(\drawCubePrefix A);
        \draw (\drawCubePrefix V)--(\drawCubePrefix B);
        \draw (\drawCubePrefix V)--(\drawCubePrefix C);
    \end{scope}
}

\title[Polynomial fillings of sublevel sets]{%
    Polynomial filling functions of sublevel sets \mbox{in \CAT{0} cube complexes}%
}

\author{Jannis Weis}
\address{\foreignlanguage{ngerman}{Fakultät für Mathematik, Karlsruher Institut für Technologie}, 76131 Karlsruhe, Germany}
\email{jannis.weis@kit.edu}

\date{5th October 2026}

\keywords{Dehn functions, \CAT{0} cube complexes, Morse theory, Higman--Thompson groups}
\subjclass[2020]{Primary 20F65; Secondary 20J05, 20F69}

\newcommand{\CAT}[1]{\ensuremath{\operatorname{CAT}(#1)}}
\newcommand{\Dehn}{\delta}
\newcommand{\symdiff}{\mathbin{\triangle}}

\newcommand{\augm}[1]{\widetilde{#1}}
\newcommand{\suspension}[1]{\Sigma #1}
\newcommand{\simplink}{\operatorname{Lk}}
\newcommand{\lkdown}{\operatorname{Lk}^{\downarrow}}
\newcommand{\simpstar}{\operatorname{St}}
\NewDocumentCommand{\conemap}{m o}{\mathrm{cone}_{#1\IfValueT{#2}{,#2}}}
\newcommand{\vertices}[1]{V(#1)}

\newcommand{\cubecell}[2]{Q(#1;#2)}
\newcommand{\heightmodule}[2]{D_{#1,#2}}
\newcommand{\topboundary}{\partial^{+}}
\newcommand{\bottomboundary}{\partial^{-}}
\newcommand{\sublevel}[2]{#1_{\leq #2}}
\newcommand{\superlevel}[2]{#1_{\geq #2}}
\newcommand{\linkshift}{\Theta}
\newcommand{\linkcontraction}{K}
\newcommand{\morsepush}{P}

\ifLuaTeX%
    \newcommand{\indexKernAdjust}{\!}
\else%
    \newcommand{\indexKernAdjust}{}
\fi

\newcommand{\overlapconstant}{\delta}
\newcommand{\altoverlapconst}{\delta'}
\newcommand{\densityset}[1]{E(#1)}
\newcommand{\descChoice}{\epsilon}
\newcommand{\densityhyperplanes}[1]{\mathcal{H}(#1)}
\newcommand{\dualhyperplane}[1]{\operatorname{Hyp}(#1)}
\newcommand{\hyperplanesupport}[1]{\mathcal{H}(#1)}
\newcommand{\associatedhyperplane}[2]{\mathcal{H}_{\indexKernAdjust#1}(#2)}
\newcommand{\goodpart}[2]{\mathord{\operatorname{good}}_{#2}(#1)}
\newcommand{\badpart}[2]{\mathord{\operatorname{bad}}_{#2}(#1)}
\newcommand{\principalpart}[2]{p_{#2}(#1)}
\newcommand{\residuepart}[2]{r_{#2}(#1)}
\newcommand{\choicevertices}[1]{\mathcal{V}^{\pm}(#1)}
\newcommand{\uppervertices}[1]{\mathcal{V}^{+}(#1)}
\newcommand{\lowervertices}[1]{\mathcal{V}^{-}(#1)}
\newcommand{\defectset}[2]{\mathcal{D}_{#2}(#1)}
\newcommand{\defectweight}[2]{\operatorname{Def}_{#2}(#1)}

\newcommand{\expect}{\mathbb{E}}
\newcommand{\totalorders}[1]{\Omega(#1)}
\newcommand{\leastelement}[2]{\mu_{#1}(#2)}
\newcommand{\precOrd}{\mathord{\prec}}

\newcommand{\sfA}{\ensuremath{\mathsf A}}
\newcommand{\rootedtree}[1]{\calT_{\indexKernAdjust#1}}
\newcommand{\FScomplex}[3]{X^{#1}_{#2,#3}}
\newcommand{\RNcomplex}[2]{X_{#1}(#2)}

\newcommand{\vertexrank}[1]{\operatorname{rk}(#1)}
\newcommand{\heightrank}[1]{\mathord{\operatorname{rk}}_{#1}}
\newcommand{\matchingcomplex}[1]{\mathcal{M}(#1)}
\newcommand{\pathhypergraph}[2]{P^{(#1)}_{#2}}
\newcommand{\pathgraph}[1]{P_{#1}}
\newcommand{\cyclehypergraph}[2]{C^{(#1)}_{#2}}
\newcommand{\cyclegraph}[1]{C_{#1}}
\NewDocumentCommand{\completehypergraph}{s m m}{%
    K^{(#2)}\IfBooleanTF{#1}{(#3)}{_{#3}}%
}
\NewDocumentCommand{\completeorderedhypergraph}{s m m}{%
    \tilde{K}^{(#2)}\IfBooleanTF{#1}{(#3)}{_{#3}}%
}
\newcommand{\orderedmatchingcomplex}[1]{\mathcal{M}^{\mathrm{o}}(#1)}
\NewDocumentCommand{\powerset}{o m}{%
    \IfValueTF{#1}{%
        \calP_{\indexKernAdjust#1}(#2)%
    }{%
        \calP(#2)%
    }%
}

\begin{document}

\begin{abstract}
    We give a criterion ensuring that sufficiently high sublevel sets of a discrete Morse function
    on a \CAT{0} cube complex have polynomial integral homological filling functions.
    As one application, we verify this criterion for the Stein--Farley complexes of the Higman--Thompson groups $F_{d,r}$, $T_{d,r}$, and $V_{d,r}$.
    Consequently, these groups have polynomial homological filling functions in every degree $k\geq2$.
\end{abstract}

\maketitle

\overfullrule=1cm

\section{Introduction}%
\label{sec:introduction}

For a finitely presented group $G$, the \emph{Dehn function} $\Dehn_G$ measures the difficulty of
solving the word problem in $G$ using a brute-force approach.
Geometrically, it records the largest area needed to fill a loop of bounded length in the universal cover of a finite presentation complex for $G$.
This viewpoint extends naturally to higher dimensions.
Replacing loops by $k$\=/dimensional spheres gives the \emph{homotopical higher Dehn function} $\Dehn_G^k$.
In this case, one requires the group to be of type $\F_{k+1}$, meaning that it admits a classifying space with finite $(k+1)$\=/skeleton or, equivalently, a cocompact cellular action on a $k$\=/connected CW complex $X$, and one considers the corresponding filling problem in $X$.
The asymptotic equivalence class of $\Dehn_G^k$ is well-defined and is a quasi-isometry invariant for groups of type $\F_{k+1}$~\cite{alonso+wang+pride}.

The $k$\=/dimensional \emph{homological filling function} $\filling{X}{k}$ of a $(k-1)$\=/acyclic CW complex $X$ is obtained by replacing spheres with integral cellular $(k-1)$\=/cycles
and considering fillings by cellular $k$\=/chains, where the size of a chain is measured using the $\ell^1$-norm on the cellular chain complex $C_\ast(X)$.
Thus, $\filling{X}{k}(n)$ measures the largest filling size needed for a cycle of norm at most $n$.
The $k$\=/dimensional \emph{homological Dehn function} $\filling{G}{k}$ of a group $G$ can be defined more generally
as $\filling{X}{k}$ for a $(k-1)$\=/acyclic cocompact $G$\=/CW complex $X$.
If $G$ admits such a complex $X$, it is said to be of type $\FH_k$.
In particular, $\F_k$ implies $\FH_k$, so both $\Dehn_G^{k-1}$ and $\filling{G}{k}$ are defined for groups of type $\F_k$.
Note that, for historical reasons, there is an index shift in the notation.
These functions are closely related.
For $k\geq4$, one has $\Dehn_G^{k-1}\approx\filling{G}{k}$, whereas for $k=3$ one only has $\Dehn_G^2\preccurlyeq\filling{G}{3}$, and the latter inequality can be strict~\cite{young}.
Consequently, bounds for $\filling{G}{k}$ yield bounds for $\Dehn_G^{k-1}$ whenever $k\geq3$.
For $k=2$, all known examples satisfy $\filling{G}{2}\preccurlyeq\Dehn_G$, and this comparison can be strict (see~\cite{Abrams2013}).

Given a discrete Morse function $h \colon X \to \bbR$ on a cube complex $X$
and $a \in \bbR$, the subcomplex $X_{\leq a}$, spanned by those
vertices $v \in X^{(0)}$ with $h(v) \leq a$, is called a \emph{sublevel set} of $X$.
A common strategy for establishing finiteness properties for a group $G$ is
to find a suitable discrete Morse function on a contractible complex $X$ such
that $G$ acts properly and cocompactly on each sublevel set and their connectivity
increases with $a$.
In these cases, the Dehn functions of $G$ agree with the filling functions
of sufficiently connected sublevel sets.
This makes filling functions of sublevel sets a natural tool for studying the (higher) Dehn functions of such groups.

To obtain upper bounds on filling functions of sublevel sets, one can first fill a cycle in the ambient complex
and then push the filling into the desired sublevel set.
Under special geometric or algebraic assumptions, pushing constructions have yielded polynomial bounds in several settings~\cites{gersten+short2002,ABDDY13,BradyKrophollerSoroko2021,kropholler+llosa-isenrich+soroko2025,ChangGarciaMigliorini2026}.
A natural goal is to employ a general procedure that pushes a filling one level at a time along descending links,
as used by Kropholler, Llosa Isenrich, and Soroko~\cite{kropholler+llosa-isenrich+soroko2025}*{Section~3.3}.
The difficulty is that a fixed multiplicative increase in norm at each step gives an upper bound that is exponential in the height of the original filling.
In their hyperbolic setting, they overcome this difficulty by choosing ambient fillings whose height is logarithmic in the boundary length.
Only logarithmically many pushing steps are then required, so the resulting filling bound is polynomial in the boundary length.

The main purpose of this paper is to provide general conditions under which a levelwise pushing procedure yields polynomial filling functions
without any requirements on the height of ambient fillings.
We call these conditions \emph{strong uniform density of descending links} and the \emph{large overlap condition} (see~\cref{def:strong-uniform-density-descending-links,def:large-overlap-descending-links}).
Roughly, the first requires descending links at sufficiently large heights to be uniformly acyclic in each fixed dimension, whereas the second requires, in an appropriate sense, increasingly large overlap between the descending links of neighbouring vertices.

Our main geometric result is the following.

\getkeytheorem*{thm:main-sublevel-fillings}{main-sublevel-fillings}

As a first example, we apply this criterion to the Higman--Thompson groups $F_{d,r}$, $T_{d,r}$, and $V_{d,r}$, where $d\geq2$ and $r\geq1$.
Their elements may be described by equivalence classes of triples $[F_+,g,F_-]$, where $F_+$ and $F_-$ are $d$\=/ary forests with $r$ roots and the same number of leaves, and $g$ is a permutation of the leaves.
For $F_{d,r}$, this permutation is always the identity; for $T_{d,r}$,
it is a cyclic permutation; and for $V_{d,r}$, no restrictions on the permutation are imposed.
We refer to~\cite{SkipperWu2023}*{Section~3.1} for a precise definition.
For $(d,r)=(2,1)$, these are Thompson's groups $F$, $T$, and~$V$.

Brown proved in~\cite{Brown1987} that the Higman--Thompson groups are of type $\F_\infty$, meaning that they are of type $\F_k$ for every $k$, by studying actions on filtered contractible complexes.
Stein later introduced simplified versions of these complexes~\cite{Stein1992}, and Farley constructed \CAT{0} cubical models for them~\cites{Farley2003,Farley2005}.
These models are now known as Stein--Farley complexes.
Crucially for our purposes, these complexes carry a natural discrete Morse function.

Guba first established polynomial upper bounds for the Dehn functions of the Thompson groups $F$, $T$, and $V$~\cites{Guba1998,Guba2000}.
He later proved that the Dehn function of $F=F_{2,1}$ is quadratic~\cite{Guba2005}, and Zhang extended this result to $F_{d,r}\cong F_{d,1}$ for $d>2$~\cite{Zhang2015}.
Migliorini proved that $T_{d,1}$ has quadratic Dehn function for every $d\geq2$~\cite{Migliorini2025}.
The Dehn function of $V$ is also conjectured to be quadratic, while the best currently known upper bound is a polynomial of degree six~\cite{lorrain2026}.
The higher-dimensional filling functions have remained less well understood.
In particular, Zaremsky asked whether the higher-dimensional Dehn functions of $F$, $T$, and $V$ all
behave like those of Euclidean space~\cite{ZaremskyOpenProblems}*{Problem~2.5}.
Although we do not answer Zaremsky's question in full,
we give an affirmative answer to the slightly weaker question of whether the higher-dimensional Dehn functions are polynomially bounded.
To do so, we show that the Stein--Farley complexes of all Higman--Thompson groups satisfy the hypotheses of~\cref{thm:main-sublevel-fillings}
and then use the fact that these groups act cocompactly and with finite stabilizers
on the sublevel sets of their Stein--Farley complexes.

\begin{thmalpha}%
    \label{thm:polynomial-stein-farley}
    Let $d\geq2$ and $r\geq1$, and let $\sfA \in \{F,T,V\}$.
    Then the Stein--Farley complex $\FScomplex{\sfA}{d}{r}$
    has strong uniform density of descending links and satisfies the large overlap condition.

    In particular, $\filling{\sfA_{d,r}}{k}$ is polynomial for every $k\geq 2$.
\end{thmalpha}

By the comparison between homological and homotopical Dehn functions discussed above, this also gives
polynomial upper bounds for $\Dehn_G^k$ for every $k\geq2$.

The same method applies to the Stein--Farley complexes of Röver--Nekrashevych groups.
Let $\rootedtree{d}$ be the rooted $d$\=/ary tree.
For each self-similar subgroup $G\leq\Aut(\rootedtree{d})$, the associated Röver--Nekrashevych group $V_d(G)$ generalizes Thompson's group $V$.
A subgroup $G$ of $\Aut(\rootedtree{d})$ is self-similar if every
restriction of an element of $G$ to a rooted subtree of $\rootedtree{d}$ is again an element of $G$.
See~\cite{SkipperZaremsky2024} for a precise definition of $V_d(G)$.
The Stein--Farley complex of $V_d(G)$ has strong uniform density of descending links and satisfies the large overlap condition for every self-similar group $G$.
The $V_d(G)$\=/action on each sublevel set is cocompact and, when $G$ is finite, has finite cell stabilizers.

\begin{thmalpha}%
    \label{thm:polynomial-RN}
    Let $d\geq2$ and let $G\leq\Aut(\rootedtree{d})$ be self-similar.
    The Stein--Farley complex $\RNcomplex{d}{G}$ has strong uniform density
    of descending links and satisfies the large overlap condition.

    In particular, if $G$ is finite,
    $\filling{V_{d}(G)}{k}$ is polynomial for every $k\geq 2$.
\end{thmalpha}

\subsection*{Strategy of proof}
The proof of~\cref{thm:main-sublevel-fillings} uses a probabilistic version of Morse pushing on the cellular chain complex.
A push operator at height $m$ is determined by a choice of contractions for the descending links at vertices of height $m$.
Under the assumptions of~\cref{thm:main-sublevel-fillings}, these can be constructed from a choice of a descending edge $e_v$ at each vertex $v$.
One observes that the resulting terms after pushing cancel whenever the choices $e_v$ and $e_w$ for adjacent vertices $v$ and $w$ determine the same hyperplane.
The large overlap between neighbouring descending links guarantees that such cancellations occur with high probability for a suitable random choice.
Estimating the expected cost then gives a choice for which the iterated pushing cost is polynomial.

The central idea of this probabilistic approach is to equip a suitable non-empty family of fillings with a probability distribution for which the expected volume is polynomially bounded.
It follows that some (non-explicit) filling in the family has polynomially bounded volume.

Probabilistic arguments based on the expected volume of random chains have previously proved useful in filling problems.
Examples include the averaging argument in the Federer--Fleming deformation theorem~\cite{Federer1960}, on which Dotterrer builds in his work on fillings in the cube~\cite{Dotterrer2015}, as well as the work of Leuzinger--Young
on filling functions of arithmetic groups~\cite{LeuzingerYoung2021}.

\subsection*{Structure}

In~\cref{sec:preliminaries}, we recall definitions of filling functions
and basic notions in simplicial complexes.
\Cref{sec:uniform-density-contractions} constructs uniformly bounded partial chain contractions of flag complexes.
In~\cref{sec:cubical-Morse-pushing}, we develop the chain-level pushing procedure and prove~\cref{thm:main-sublevel-fillings}.
Finally, in~\cref{sec:Higman-Thompson-strong-uniform-density,sec:Rover-Nekrashevych-strong-uniform-density}, we verify strong uniform density and the large overlap condition and deduce the polynomial bounds for Higman--Thompson and certain Röver--Nekrashevych groups.

\subsection*{Disclosure of AI usage}
Large language model (LLM) tools, specifically GPT~5.6, were used to produce the combinatorial arguments in~\cref{sec:Higman-Thompson-strong-uniform-density,sec:Rover-Nekrashevych-strong-uniform-density}.
To obtain these results, the LLM was given the conditions in~\cref{def:strong-uniform-density-descending-links,def:large-overlap-descending-links}, which must be verified to apply~\cref{thm:main-sublevel-fillings}, and was tasked with checking them for the Stein--Farley complexes of Higman--Thompson groups and Röver--Nekrashevych groups.
For this purpose, the LLM also formulated and proved~\cref{prop:full-descending-link-density-criterion} as a sufficient criterion for the conditions of~\cref{thm:main-sublevel-fillings}.
For readability, we place it in~\cref{sec:cubical-Morse-pushing} rather than~\cref{sec:Higman-Thompson-strong-uniform-density}.
The tools were also used for editorial checks and proofreading.
All the text in this article was written by me, and I take full responsibility for its content.

\subsection*{Acknowledgements}
I would like to thank Claudio Llosa Isenrich, Matteo Migliorini, Roman Sauer, and Eduard Schesler for many helpful discussions and comments during the preparation of this work.
This research was supported by the \foreignlanguage{ngerman}{Deutsche Forschungsgemeinschaft} (DFG, German Research Foundation) through Research Training Group 2229, project number 281869850.

\section{Preliminaries}%
\label{sec:preliminaries}

All chain complexes are taken with integer coefficients.
Differentials of chain complexes are sometimes generically written as $\partial$.
If $C$ is a based free $\bbZ$\=/module and $c=\sum_\sigma c_\sigma\sigma$,
its $\ell^1$-norm is given by $\norm{c}=\sum_\sigma\abs{c_\sigma}$.
For a chain complex $C_\ast$, its \emph{suspension} $\suspension{C_\ast}$ is the chain complex defined by $(\suspension{C_\ast})_k=C_{k-1}$ and $\partial^{\suspension{C_\ast}}_k=-\partial^C_{k-1}$.

\subsection{Homological filling functions}

\begin{defn}%
    \label{def:cellular-filling-function}
    Let $k \geq 2$ and let $X$ be a $(k-1)$\=/acyclic CW complex.
    The \emph{filling norm} of a cellular $(k-1)$\=/cycle $z$ is
    \[
        \fillvolume[X](z) \defq \inf\{\norm{b} \mid b\in C_k(X),\ \partial b=z\}.
    \]
    The \emph{$k$-dimensional homological filling function} of $X$ is
    \[
        \filling{X}{k}(n) \defq \sup\{\fillvolume[X](z) \mid z\in C_{k-1}(X),\ \partial z = 0,\ \norm{z}\leq n\}.
    \]

    Given functions $f, g \colon \bbR_{\ge 0} \to \bbR_{\ge 0}$, we write $f \preccurlyeq g$ if there are constants
    $C, D > 0$ such that $f(v) \leq C g(Cv + D) + Cv + D$ for every $v \geq 0$.
    We write $f \approx g$ if both $f \preccurlyeq g$ and $g \preccurlyeq f$ hold.
\end{defn}

\begin{defn}
    A function $f \colon \bbR_{\ge 0} \to \bbR_{\ge 0}$ is called \emph{polynomial} if there are constants $A>0$
    and $d\geq0$ such that $f(v) \leq A(1+v)^d$ for every $v \geq 0$.
    It is called \emph{superadditive} if, for every partition $n=n_1+\cdots+n_k$ of $n\in \mathbb{N}_0$ into non-negative integers, we have $f(n)\geq f(n_1)+\cdots+f(n_k)$.
\end{defn}

\begin{defn}
    Let $G$ be a group of type $\FH_n$, and let $X$ be an $(n-1)$\=/acyclic $G$\=/CW complex on which $G$ acts freely, properly, and cocompactly.
    The \emph{\Nth{n} homological Dehn function} of $G$ is the function $\filling{G}{n} \defq \filling{X}{n}$.
    This function is well-defined up to $\approx$\=/equivalence~\cite{alonso}.
\end{defn}

\subsection{Simplicial chains and flag complexes}

We regard a simplex both as a finite set of vertices and as a cell.
An ordering $[v_0,\ldots,v_j]$ of its vertices determines an orientation.
Two orderings determine the same orientation precisely when they differ by an even permutation.
If a vertex occurs twice, the resulting element of the simplicial chain module is zero.

\begin{defn}%
    \label{def:simplicial-language}
    Let $\sigma$ be a simplex of a simplicial complex $L$.
    \begin{refenum}
        \item\label{def:simplicial-language:star}
              The \emph{star} $\simpstar_L(\sigma)$ consists of the
              simplices contained in some simplex of the form
              $\sigma\cup\tau$.
        \item\label{def:simplicial-language:link}
              The \emph{link} $\simplink_L(\sigma)$ consists of the
              simplices $\tau$ which are disjoint from $\sigma$ and for which
              $\sigma\cup\tau$ is a simplex of $L$.
        \item\label{def:simplicial-language:flag}
              The complex $L$ is \emph{flag} if every finite set of
              pairwise adjacent vertices is the vertex set of a simplex.
    \end{refenum}
\end{defn}

Let $L$ be a simplicial complex.
We write $\augm{C}_*(L)$ for the augmented simplicial chain complex with $\augm{C}_{-1} = \bbZ$.
Suppose that a map $f\colon L^{(0)}\to K^{(0)}$ takes the vertex set of every simplex of $L$ to the vertex set of a simplex of $K$.
It induces a map $f_\ast \colon \augm{C}_\ast(L) \to \augm{C}_\ast(K)$ given by
$f_j([v_0,\ldots,v_j]) = [f(v_0),\ldots,f(v_j)]$.
If a tuple $\sigma = [v_0,\ldots,v_j]$ contains a repeated vertex, then we consider $\sigma = 0$ as an element of $C_j(L)$.
Let $x\in L^{(0)}$.
For every $\sigma=[v_0,\ldots,v_j]$ such that
$\{x,v_0,\ldots,v_j\}$ is a simplex of $L$,
we define the \emph{cone} $x*\sigma\defq[x,v_0,\ldots,v_j]$.
This cone construction induces homomorphisms
$\conemap{x}[j] \colon \augm{C}_j(\simpstar_L(x)) \to \augm{C}_{j+1}(L)$, where $\conemap{x}[-1](1) = [x]$.
They satisfy $\partial \conemap{x}[j]+\conemap{x}[j-1]\partial=\id$ on $\augm{C}_*(\simpstar_L(x))$.
Moreover, they do not increase the $\ell^1$-norm.

\begin{defn}%
    \label{def:vertex-support}
    Let $L$ be a simplicial complex.
    For $c = \sum_\sigma c_\sigma \sigma \in C_j(L)$, its \emph{support} is defined as $\supp(c) \defq \{\sigma \mid c_\sigma \neq 0\}$.
    Its \emph{vertex support} $\vertices{c}$ is the union of the vertex sets of the simplices in $\supp(c)$.
\end{defn}

These sets satisfy $\size{\supp(c)} \leq \norm{c}$ and
$\size{\vertices{c}} \leq (j+1)\size{\supp(c)} \leq (j+1)\norm{c}$ for every integral $j$-chain $c$.
The same sets are defined analogously for cubical complexes.

\begin{defn}%
    \label{def:compatibility}
    Let $L$ be a simplicial complex and let $\sigma$ be a simplex of $L$.
    Two vertices are \emph{compatible} in $L$ if they are equal or span
    an edge, and they are \emph{incompatible} otherwise.
    A vertex is \emph{compatible} with $\sigma$ if it
    is compatible with every vertex of $\sigma$.
\end{defn}

If $L$ is flag, a vertex $x$
is compatible with a simplex $\sigma$ precisely when $\sigma$ belongs to $\simpstar_L(x)$.

\begin{defn}%
    \label{def:star-retraction}
    Let $L$ be a flag complex and let $e\in L^{(0)}$.
    Define a map on the
    vertices of $L$ by
    \[
        r_e(v)=\begin{cases}
        v,&\text{if $v$ is compatible with $e$},\\
        e,&\text{if $v$ is incompatible with $e$}.
        \end{cases}
    \]
\end{defn}

\begin{lem}%
    \label{lem:star-retraction}
    The map in~\cref{def:star-retraction}
    extends to a well-defined simplicial map $r_e\colon L\to\simpstar_L(e)$,
    which restricts to the identity on $\simpstar_L(e)$.
    It is therefore a simplicial retraction.
\end{lem}
\begin{proof}
    Let $\sigma$ be a simplex.
    The distinct vertices of $r_e(\sigma)$ are $e$ together with vertices of $\sigma$ which
    are adjacent to $e$, and hence the images of the vertices of $\sigma$ are pairwise adjacent.
    Because $L$ is flag, they span a simplex in $\simpstar_L(e)$.
    Note that every vertex of $\simpstar_L(e)$ is compatible with $e$ and is therefore fixed by $r_e$.
    Hence, $r_e$ is a retraction.
\end{proof}

We also denote the induced map on chain complexes by $r_e$.
Because it is a retraction, it satisfies $\norm{r_e(c)}\leq\norm{c}$.

\section{Uniform contractions of dense flag complexes}%
\label{sec:uniform-density-contractions}

The following notion is a variation of the \emph{uniform density}
condition due to Belk--Forrest~\cite{BelkForrest2019}.
It gives a combinatorial criterion for the cellular chain complex of a simplicial complex to
admit a bounded chain contraction.

\begin{defn}%
    \label{def:uniform-density}
    Let $b\in\bbN_0$ and let $L$ be a flag complex.
    A subset $E\subseteq L^{(0)}$ is called \emph{$b$\=/dense} in $L$ if every vertex of $L$ is incompatible with at most $b$ vertices of $E$.
    We say that $L$ is uniformly \emph{$b$\=/dense} if $L^{(0)}$ is $b$\=/dense in $L$.
\end{defn}

If the vertex set $\vertices{\sigma}$ of a simplex in $L$ is $b$\=/dense, it is also referred to as a \emph{ground}~\cite{BelkForrest2019}*{Definition~4.8}.
However, we do not make use of this notion.

\begin{defn}%
    \label{def:R-good-density}
    Let $R>0$ and $b\in\bbN_0$.
    Let $L$ be a flag complex and let $E\subseteq L^{(0)}$ be $b$\=/dense in $L$.
    We say that $E$ is \emph{$(R,b)$\=/good} if $\size{E}>Rb$.
\end{defn}

\begin{thm}[{\cite{BelkForrest2019}*{Theorem~4.11}}]%
    \label{thm:Belk-Forrest-density}
    Let $k\geq0$ and $b\geq1$.
    Let $L$ be a uniformly $b$\=/dense flag complex.
    If $L^{(0)}$ is $((b+1)(k+1),b)$\=/good, then $L$ is $k$\=/connected.
\end{thm}

The following~\lcnamecref{lem:quantitative-density-homotopy} can be considered a quantitative homological analogue of Belk and Forrest's~\cref{thm:Belk-Forrest-density}.

\begin{lem}%
    \label{lem:quantitative-density-homotopy}
    Fix $p\geq0$.
    There exist non-negative constants
    $K_{-1}, K_0,\ldots, K_p$ and a constant $R_p>0$ with the following property.
    Let $L$ be a flag complex with a subset $E \subseteq L^{(0)}$ which is $b$\=/dense and $(R_p,b)$\=/good for some $b \in \bbN_0$.
    For every $e\in E$, there are homomorphisms
    $H_{e,j}\colon\augm{C}_j(L)\to\augm{C}_{j+1}(L)$ for $-1\leq j\leq p$ such that
    \begin{refenum}
        \item\label{lem:quantitative-density-homotopy:homotopy}
              $\partial H_{e,j}+H_{e,j-1}\partial=\id - r_{e,j}$ for $0\leq j\leq p$.
        \item\label{lem:quantitative-density-homotopy:norm}
              $\norm{H_{e,j}(\sigma)}\leq K_j$ for every $j$-simplex $\sigma$.
        \item\label{lem:quantitative-density-homotopy:vanishing}
              $H_{e,j}(\sigma)=0$ if $\sigma$ belongs to $\simpstar_L(e)$.
    \end{refenum}
\end{lem}
\begin{proof}
    Set $K_{-1}=0$ and $K_j=2+(j+1) K_{j-1}$ for $0\leq j\leq p$.
    We shall prove the result with $R_p=\max_{0\leq j\leq p}(j+1) K_j$.

    Fix $e\in E$ and set $H_{e,-1}=0$.
    For the inductive step, let $0\leq j\leq p$.
    Suppose that $H_{e,i}$ has been constructed up to degree $j-1$.
    Let $\sigma$ be a $j$-simplex and put
    $\beta_\sigma=\sigma-r_e(\sigma)-H_{e,j-1}(\partial\sigma)$.
    Then
    \begin{align*}
        \partial \beta_\sigma &= (\id - r_{e,j-1})(\partial \sigma) - \partial H_{e,j-1}(\partial \sigma) \\
                              &= (\partial H_{e,j-1} + H_{e,j-2}\partial)(\partial \sigma) - \partial H_{e,j-1}(\partial \sigma) = 0.
    \end{align*}
    Hence, $\beta_\sigma$ is a cycle.
    Furthermore, $\norm{\beta_\sigma}\leq 2+(j+1) K_{j-1}= K_j$, so
    \begin{equation}
        \size{\vertices{\beta_\sigma}}\leq(j+1) K_j \leq R_p.
    \end{equation}
    Each vertex in $\vertices{\beta_\sigma}$ is incompatible with at most $b$ vertices of $E$, so at most $R_pb$ vertices of $E$ are incompatible with some vertex of $\vertices{\beta_\sigma}$.
    Since $E$ is $(R_p,b)$\=/good, there is a vertex $x_\sigma\in E$ compatible with all of $\vertices{\beta_\sigma}$.
    Because $L$ is flag, this implies that $\beta_\sigma$ is supported in $\simpstar_L(x_\sigma)$.

    Define $H_{e,j}(\sigma)=x_\sigma*\beta_\sigma$ and extend linearly.
    Then $\partial H_{e,j}(\sigma) = \beta_\sigma$, so~\localref{lem:quantitative-density-homotopy:homotopy} holds in degree $j$.
    Also,~\localref{lem:quantitative-density-homotopy:norm} follows immediately, because $\norm{x_\sigma * \beta_\sigma} \leq \norm{\beta_\sigma} \leq K_j$.

    If $\sigma$ belongs to $\simpstar_L(e)$, then $r_e(\sigma)=\sigma$
    and every cell in $\partial \sigma$ is also contained in $\simpstar_L(e)$.
    Hence, induction on the dimension of $\sigma$ gives $H_{e,j-1}(\partial \sigma) = 0$,
    so $\beta_\sigma = 0$ and $H_{e,j}(\sigma) = 0$, proving~\localref{lem:quantitative-density-homotopy:vanishing}.
\end{proof}

The homotopy $H_e$ measures the difference between the identity and the retraction $r_e$.
Composing the retraction with the cone operator on its image turns this into a contraction of the augmented chain complex through degree $p$.

\begin{defn}%
    \label{def:density-contraction}
    Let $p\geq0$ and
    let $L$ be a flag complex with a subset $E \subseteq L^{(0)}$ which is $b$\=/dense and $(R_p,b)$\=/good for some $b \in \bbN_0$.
    Let $H_e$ be the maps obtained from~\cref{lem:quantitative-density-homotopy}.
    For $e\in E$, define
    \[
        F_{e,j} \defq H_{e,j} + \conemap{e}[j] \circ r_{e,j} \colon \augm{C}_j(L)\longrightarrow\augm{C}_{j+1}(L),
    \]
    for $-1\leq j\leq p$.
\end{defn}

\begin{cor}%
    \label{cor:quantitative-density-contraction}
    Let $p\geq0$ and $b\in\bbN_0$,
    and let $L$ be a flag complex with a $b$\=/dense, $(R_p,b)$\=/good subset $E \subseteq L^{(0)}$.
    Then for every $e\in E$, one has
    \begin{refenum}
        \item\label{cor:quantitative-density-contraction:contraction}
              $\partial F_{e,j} + F_{e,j-1}\partial = \id$ for $-1 \leq j \leq p$;
        \item\label{cor:quantitative-density-contraction:basis-norm}
              $\norm{F_{e,j}(\sigma)} \leq K_j+1$ for every $j$-simplex $\sigma$;
        \item\label{cor:quantitative-density-contraction:cone-formula}
              $F_{e,j}(\sigma)=e*\sigma$ if $e$ is compatible with $\sigma$.
    \end{refenum}
\end{cor}
\begin{proof}
    Ad~\localref{cor:quantitative-density-contraction:contraction}.
    This follows immediately from~\cref{lem:quantitative-density-homotopy:homotopy} and the identity $\partial \conemap{e}[j] + \conemap{e}[j-1] \partial=\id$.

    Ad~\localref{cor:quantitative-density-contraction:basis-norm}.
    Use~\cref{lem:quantitative-density-homotopy:norm} and the fact that $r_{e,j}$ and $\conemap{e}[j]$ are norm non-increasing.

    Ad~\localref{cor:quantitative-density-contraction:cone-formula}.
    Compatibility implies that $\sigma$ belongs to
    $\simpstar_L(e)$ and that $r_{e,j}(\sigma)=\sigma$.
    Consequently,~\cref{lem:quantitative-density-homotopy:vanishing} gives $H_{e,j}(\sigma)=0$, while
    the construction of $\conemap{e}$ gives $\conemap{e}[j] \circ r_{e,j}(\sigma)=e*\sigma$.
\end{proof}

\section{Cubical Morse pushing}%
\label{sec:cubical-Morse-pushing}

The goal of this~\lcnamecref{sec:cubical-Morse-pushing} is to prove~\cref{thm:main-sublevel-fillings}.
We start by recalling some basic notions of discrete Morse theory due to Bestvina and Brady~\cite{BestvinaBrady1997}
in the context of cube complexes.

\begin{defn}%
    \label{def:discrete-Morse-function}
    Let $X$ be a cube complex.
    A \emph{discrete Morse function} on $X$ is a continuous function $h\colon X\to\bbR$ that is affine on every cube, takes integer values on vertices, and changes by one along every edge.
\end{defn}

The last condition is usually not part of the definition.
However, one can usually arrange a Morse function to satisfy it, so it is no real restriction.
Given a Morse function on $X$, every cube has a unique vertex of maximal height and a unique vertex of minimal height.

\begin{defn}%
    \label{def:sublevel-complex}
    For $a\in\bbZ$, the \emph{sublevel complex} $\sublevel{X}{a}$ is the cubical subcomplex consisting of the cubes whose maximal vertex has height at most $a$.
    The \emph{superlevel complex} $\superlevel{X}{a}$ is the cubical subcomplex consisting of the cubes whose minimal vertex has height at least $a$.
\end{defn}

Throughout this section, let $X$ be a \CAT{0} cube complex equipped with a discrete Morse function $h$.
An edge $a$ adjacent to a vertex $v$ is called \emph{descending at $v$} if $v$ maximizes the value of $h$ on $a$.

\begin{defn}%
    \label{def:descending-link}
    Let $v$ be a vertex of $X$.
    The \emph{descending link} $\lkdown(v)$ is the simplicial complex whose vertices are the descending edges adjacent to $v$.
    A finite collection of these vertices is a simplex if the corresponding edges belong to a common cube of $X$.

    For ease of notation, we identify each vertex $a\in\lkdown(v)^{(0)}$
    with the corresponding descending edge at $v$.
\end{defn}

By Gromov's link condition, the link of every vertex of a \CAT{0} cube complex is flag (see~\cite{Niblo1998}*{Section~2}).
The descending link is the full subcomplex of the vertex link spanned by the descending edges and is therefore flag as well.

\subsection{Link transport and the height grading}

\begin{defn}%
    \label{def:descending-edge-transport}
    Let $v$ be a vertex of $X$.
    For a descending edge $a \in \lkdown(v)^{(0)}$, we denote the other endpoint of $a$ by $v/a$.

    Let $\sigma$ be a simplex of $\lkdown(v)$ and let $a\in\lkdown(v)^{(0)}$ be compatible with $\sigma$, that is, $\vertices{\sigma}\cup\{a\}$ spans a simplex of $\lkdown(v)$.
    For $b\in \vertices{\sigma}\setminus\{a\}$, the edges $a$ and $b$ lie in a square.
    The edge in this square which is opposite to $b$ and incident to $v/a$ is denoted by $b/a$.
    The other endpoint of $b/a$ opposite to $v/a$ has height $h(v)-2$, so $b/a$ is descending at $v/a$.
    Consequently,
    \[
        \sigma/a\defq\{b/a\mid b\in \vertices{\sigma}\setminus\{a\}\}
    \]
    is a simplex of $\lkdown(v/a)$.
    If $\sigma=[a_1,\ldots,a_k]$ is ordered, then
    $[a_1/a,\ldots,a_k/a]$ defines an order on $\sigma/a$, where we omit $a/a$ if it occurs.
    Write $\sigma\setminus a_i$ for the ordered simplex obtained by omitting $a_i$.
    Then
    \[
        (\sigma\setminus a_i)/a_i =[a_1/a_i,\ldots,a_{i-1}/a_i,a_{i+1}/a_i,\ldots,a_k/a_i].
    \]
\end{defn}

\begin{defn}%
    \label{def:oriented-cube-notation}
    Let $\sigma=[a_1,\ldots,a_k]$ be an ordering of the vertices of a simplex of $\lkdown(v)$.
    The corresponding cube has a parametrization by $[0,1]^k$ in which $(1,\ldots,1)$ represents $v$ and the path obtained by decreasing the \Nth{i} coordinate parametrizes the edge $a_i$.
    We give the cube the orientation induced by the standard orientation of $[0,1]^k$ and denote the resulting basis element of $C_k(X)$ by
    $\cubecell{v}{\sigma}=\cubecell{v}{a_1,\ldots,a_k}$.
    Reordering $\sigma$ by an odd permutation changes the sign of this basis element.
    For $k=0$, we put $\cubecell{v}{\emptyset}=v$.
    We extend this notation to tuples of pairwise compatible descending edges by setting $\cubecell{v}{a_1,\ldots,a_k}=0$ whenever an edge occurs more than once.
\end{defn}

\begin{rem}%
    \label{rem:transported-faces}
    Let $\sigma=[a_1,\ldots,a_k]$ be an ordering of the vertices of a simplex of $\lkdown(v)$.
    For every $1\leq i\leq k$, the vertex $a_i$ is compatible with every vertex of $\sigma\setminus a_i$ and does not itself belong to this simplex, and
    $\vertices{\sigma\setminus a_i}\cup\{a_i\}=\vertices{\sigma}$.
    Consequently, $(\sigma\setminus a_i)/a_i$ is a well-defined ordered simplex of $\lkdown(v/a_i)$, possibly the empty simplex, and
    $\cubecell{v/a_i}{(\sigma\setminus a_i)/a_i}$ is a well-defined basis cube of $C_{k-1}(X)$.
    See~\cref{fig:cube-transfer} for an example.
    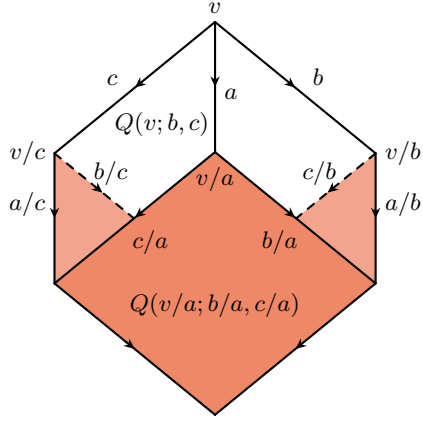
\begin{figure}
        \centering
        \begin{tikzpicture}[
        3d view={-45}{55},
        line join=round,
        line cap=round,
        font=\small,
    ]

    \def\LX{3.0}
    \def\LY{3.0}
    \def\LZ{3.0}

    \cubeCoordinates

    \begin{scope}[light Red,opacity=0.75]
        \fill[light Red] (O)--(BC)--(C)--(AC)--cycle;

        \fill[light Red] (O)--(AB)--(B)--(BC)--cycle;
    \end{scope}

    \begin{scope}[thick,dashed,every path/.append style={descending edge=0.3}]
        \draw (BC)--(O);
        \draw (B)--(BC);
        \draw (C)--(BC);
    \end{scope}

    \fill[light Red,opacity=1] (O)--(AB)--(A)--(AC)--cycle;

    \begin{scope}[thick,every path/.append style={descending edge}]
        \draw (AB)--(O);
        \draw (A)--(AB);
        \draw (A)--(AC);
        \draw (AC)--(O);
        \draw (B)--(AB);
        \draw (C)--(AC);
        \draw (V)--(A);
        \draw (V)--(B);
        \draw (V)--(C);
    \end{scope}

    \node[above] at (V) {$v$};
    \node[left] at (C) {$v/c$};
    \node[right] at (B) {$v/b$};
    \node[below=2pt] at (A) {$v/a$};

    \path
    (current bounding box.south west) coordinate (bbSW)
    (current bounding box.north east) coordinate (bbNE);

    \pgfresetboundingbox
    \path[use as bounding box] (bbSW) rectangle (bbNE);

    \node[above left] at ($(V)!0.55!(C)$) {$c$};
    \node[above right] at ($(V)!0.55!(B)$) {$b$};
    \node[right] at ($(V)!0.55!(A)$) {$a$};

    \node[below left=-2pt] at ($(A)!0.55!(AB)$) {$b/a$};
    \node[below right=-2pt] at ($(A)!0.55!(AC)$) {$c/a$};

    \node[above right] at ($(B)!0.55!(AB)$) {$a/b$};
    \node[above=2pt] at ($(B)!0.35!(BC)$) {$c/b$};

    \node[above left] at ($(C)!0.55!(AC)$) {$a/c$};
    \node[above=2pt] at ($(C)!0.35!(BC)$) {$b/c$};

    \node[below] at ($(AB)!0.5!(AC)$) {$\cubecell{v/a}{b/a,c/a}$};
    \node[below right,xshift=-8pt,yshift=-3pt] at ($(V)!0.55!(C)$) {$\cubecell{v}{b,c}$};

\end{tikzpicture}

        \caption{%
            \label{fig:cube-transfer}
            \label{fig:cube-bottom-boundary}
            A cube $Q = \cubecell{v}{a,b,c}$. Arrows on edges indicate descending directions.
            The frontmost face is the one opposite $\cubecell{v}{b,c}$.
            One can think of this face as the result of
            moving the face $\cubecell{v}{b,c}$ along the edge $a$.
            The shaded area is $\bottomboundary Q$, and the remaining faces are $\topboundary Q$.
        }
    \end{figure}
\end{rem}

\begin{defn}%
    \label{def:height-grading}
    For $m\in\bbZ$ and $k\geq0$, let $\heightmodule{m}{k}$ be the free submodule of $C_k(X)$ generated by the cubes whose maximal vertex has height equal to $m$.
    Set $\heightmodule{m}{-1}=0$.
    Let $\heightmodule{v}{k}$ be the free submodule generated by the cubes whose maximal vertex is $v$.
    Then
    \[
        C_k(X) = \;\bigoplus_{\mathclap{m\in\bbZ}}\;\heightmodule{m}{k}, \qquad \heightmodule{m}{k}= \;\bigoplus_{\mathclap{h(v)=m}}\;\heightmodule{v}{k}.
    \]
    This makes $C_*(X)$ a bigraded module $\heightmodule{\ast}{\ast}$.

    Define homomorphisms on the basis specified in~\cref{def:oriented-cube-notation} by
    \begin{align*}
        \topboundary\cubecell{v}{\sigma}    &\defq \sum_{i=1}^{k}(-1)^{i-1} \cubecell{v}{\sigma\setminus a_i}, \\
        \bottomboundary\cubecell{v}{\sigma} &\defq \sum_{i=1}^{k}(-1)^i \cubecell{v/a_i}{(\sigma\setminus a_i)/a_i}.
    \end{align*}
\end{defn}

\begin{rem}
    The two summands in~\cref{def:height-grading} collect the faces of $\cubecell{v}{\sigma}$ containing the maximal vertex and the faces opposite to it, respectively.
    Their bidegrees are
    $\topboundary\colon\heightmodule{m}{k}\to\heightmodule{m}{k-1}$ and
    $\bottomboundary\colon\heightmodule{m}{k}\to\heightmodule{m-1}{k-1}$.
\end{rem}

\begin{lem}%
    \label{lem:cubical-boundary-splitting}
    The maps $\topboundary$ and $\bottomboundary$ have the following properties.
    \begin{refenum}
        \item\label{lem:cubical-boundary-splitting:sum}
              $\partial=\topboundary+\bottomboundary$.
        \item\label{lem:cubical-boundary-splitting:squares}
              $(\topboundary)^2=(\bottomboundary)^2=0$.
        \item\label{lem:cubical-boundary-splitting:mixed}
              $\topboundary\bottomboundary+\bottomboundary\topboundary=0$.
        \item\label{lem:cubical-boundary-splitting:norm}
              $\norm{\topboundary x} \leq k \norm{x}$ and $\norm{\bottomboundary x} \leq k \norm{x}$ for every $x \in \heightmodule{m}{k}$.
    \end{refenum}
\end{lem}
\begin{proof}
    It is clear from the definition of $\topboundary$ and $\bottomboundary$ that~\localref{lem:cubical-boundary-splitting:sum,lem:cubical-boundary-splitting:norm} hold.
    The terms of $\partial^2 = (\topboundary + \bottomboundary)^2$ which preserve the height of the maximal vertex give $(\topboundary)^2=0$, while those which lower it by two give $(\bottomboundary)^2=0$, so~\localref{lem:cubical-boundary-splitting:squares} holds.
    The remaining terms lower the height by one and give~\localref{lem:cubical-boundary-splitting:mixed}.
\end{proof}

\begin{lem}%
    \label{lem:link-cube-map}
    The maps $\linkshift_{v,k}\colon\augm{C}_{k-1}(\lkdown(v))\longrightarrow\heightmodule{v}{k}$
    sending a simplex $\sigma$ to $\cubecell{v}{\sigma}$
    are isometric isomorphisms with respect to the $\ell^1$\=/norm and satisfy
    $\linkshift_{v,k-1}\partial=\topboundary\linkshift_{v,k}$ for $k\geq1$.
\end{lem}
\begin{proof}
    The basis correspondence follows from~\cref{def:descending-link,def:oriented-cube-notation}.
    The boundary identity is the formula defining $\topboundary$.
\end{proof}

\begin{cor}%
    \label{cor:descending-link-chain-identification}
    The homomorphisms $(-1)^k\linkshift_{v,k}$ define an isometric isomorphism of chain complexes
    \(\suspension{\augm{C}_*(\lkdown(v))} \longrightarrow (\heightmodule{v}{\ast},\topboundary)\).
\end{cor}

\subsection{Strong uniform density and large overlap}

Every edge $e$ in a \CAT{0} cube complex defines a \emph{hyperplane} in $X$.
The defining equivalence relation on edges is generated by declaring that $e \sim e'$ whenever $e$ and $e'$ are
opposite edges of a square in $X$.
The hyperplane $\dualhyperplane{e}$ dual to $e$ is spanned
by the midcubes whose corresponding edges in $X$ are equivalent to $e$.
The subcomplex spanned by the cubes in $X$ that contain a midcube belonging to $\dualhyperplane{e}$ is called the \emph{carrier} of $\dualhyperplane{e}$.
For further information on hyperplanes in \CAT{0} cube complexes, see~\cite{Niblo1998} and~\cite{Sageev1995}.

\begin{lem}%
    \label{lem:hyperplane-transport}
    The following statements hold.
    \begin{refenum}
        \item\label{lem:hyperplane-transport:injective}
              Distinct edges incident to the same vertex have distinct dual hyperplanes.
        \item\label{lem:hyperplane-transport:transport}
              Let $a$ be a descending edge at $v$ and put $w=v/a$.
              If $e_v$ and $e_w$ are descending edges at $v$ and $w$, respectively, then
              $\dualhyperplane{e_v}=\dualhyperplane{e_w}$ if and only if
              $e_v$ is distinct from and compatible with $a$, and $e_w=e_v/a$.
    \end{refenum}
\end{lem}
\begin{proof}
    Ad~\localref{lem:hyperplane-transport:injective}.
    The carrier of a hyperplane is a convex product subcomplex and contains a unique edge dual to that hyperplane at each of its vertices (see~\cite{Sageev1995}*{Theorems~4.10 and~4.13}).

    Ad~\localref{lem:hyperplane-transport:transport}.
    For the reverse implication, the edges $e_v$ and $e_v/a$ are opposite edges of a square and are therefore dual to the same hyperplane.
    Conversely, suppose that $\dualhyperplane{e_v}=\dualhyperplane{e_w}$.
    The edge $a$ is not descending at $w$, whereas $e_w$ is descending there, so $a \neq e_w$.
    By~\localref{lem:hyperplane-transport:injective}, their dual hyperplanes are distinct.
    Both endpoints of $a$ lie in the carrier of $\dualhyperplane{e_v}$.
    Convexity and the product structure of the carrier supply a square in which $e_v$ and $e_w$ are opposite edges.
    Hence, $e_v$ is distinct from and compatible with $a$, and $e_w=e_v/a$.
\end{proof}

\begin{defn}[Strong uniform density of descending links]%
    \label{def:strong-uniform-density-descending-links}
    Let $a_0$ be an integer.
    The descending links of $X$ satisfy \emph{strong uniform density (above height $a_0$)} if there exist positive integers $N_m,b_m$ for $m\geq a_0$ with
    \begin{condrefenum}
        \item\label{eq:strong-uniform-density-asymptotic}
              $\lim_{m \to \infty} b_m/N_m = 0$
    \end{condrefenum}
    and for every vertex $v \in \superlevel{X}{a_0}^{(0)}$, there is a finite set
    $\densityset{v}\subseteq\lkdown(v)^{(0)}$ satisfying the following property.
    \begin{condrefenum}[resume]
        \item\label{def:strong-uniform-density-descending-links:density}
              If $h(v) = m$, then $\size{\densityset{v}}\geq N_m$ and
              $\densityset{v}$ is $b_m$\=/dense in $\lkdown(v)$.
    \end{condrefenum}
\end{defn}

Whenever the descending links satisfy strong uniform density, we fix sequences $(N_m)_{m \geq a_0}$ and $(b_m)_{m \geq a_0}$, and one family $(\densityset{v})_{v\in \superlevel{X}{a_0}^{(0)}}$.

\begin{defn}[Large overlap condition]%
    \label{def:large-overlap-descending-links}
    Suppose that the descending links of $X$ satisfy strong uniform density above height $a_0$.
    Let $\densityhyperplanes{v} \defq \{\dualhyperplane{e}\mid e\in\densityset{v}\}$.

    The descending links of $X$ satisfy the \emph{large overlap condition} if there exists a constant $\overlapconstant>0$ such that, whenever
    $v, w \in \superlevel{X}{a_0}^{(0)}$ are the endpoints of an edge of $X$ and $m=\max\{h(v),h(w)\}$, one has
    \[
        \frac{\size{\densityhyperplanes{v}\symdiff\densityhyperplanes{w}}}
        {\size{\densityhyperplanes{v}\cup\densityhyperplanes{w}}}
        \leq\frac{\overlapconstant}{m},
    \]
    where $\symdiff$ denotes the symmetric difference of sets.
\end{defn}

Whenever the large overlap condition holds, we fix a corresponding constant $\overlapconstant$.
If the choice of Morse function is clear from context, we also say
that $X$ satisfies the large overlap condition.

\begin{rem}%
    \label{rem:descending-links-acyclic}
    Strong uniform density together with~\cref{cor:quantitative-density-contraction} implies that
    for every $p$, there exists some $m_p$ such that all descending links $\lkdown(v)$ with $h(v) \geq m_p$
    are uniformly $p$\=/acyclic.

    The inequality in the large overlap condition can be reformulated
    as
    \[
        \frac{\size{\densityhyperplanes{v} \cap \densityhyperplanes{w}}}
        {\size{\densityhyperplanes{v}\cup\densityhyperplanes{w}}}
        \geq 1 - \frac{\overlapconstant}{m},
    \]
    which is where the name comes from.
    We choose the alternative formulation in~\cref{def:large-overlap-descending-links}, since this is how it is used in the proof of~\cref{thm:main-sublevel-fillings}.
\end{rem}

The following criterion is often convenient for verifying strong uniform density and the large overlap condition.

\begin{prop}%
    \label{prop:full-descending-link-density-criterion}
    Let $a_0$ be an integer.
    Suppose that, for every $m\geq a_0$, the descending links at vertices of height $m$ all have the same finite number of vertices $N_m > 0$.
    Assume that there exist a constant $\altoverlapconst>0$ and positive integers $b_m$ for $m\geq a_0$ such that
    \begin{condrefenum}
        \item\label{prop:full-descending-link-density-criterion:rate}
              \(\displaystyle
              \frac{b_m+\abs{N_m-N_{m-1}}}{N_m}
              \leq\frac{\altoverlapconst}{m}
              \)
              if $m>a_0$;
        \item\label{prop:full-descending-link-density-criterion:density}
              $\lkdown(v)$ is uniformly $b_{m}$\=/dense if $h(v) = m \geq a_0$.
    \end{condrefenum}
    Then the descending links of $X$ satisfy strong uniform density and the large overlap condition above height $a_0$
    with $\densityset{v}\defq\lkdown(v)^{(0)}$.
\end{prop}
\begin{proof}
    Condition~\localref{prop:full-descending-link-density-criterion:density} is~\cref{def:strong-uniform-density-descending-links:density} for $\densityset{v}\defq\lkdown(v)^{(0)}$.
    Moreover,~\localref{prop:full-descending-link-density-criterion:rate} gives
    $b_m/N_m\leq\altoverlapconst/m$, and hence $\lim_{m \to \infty} b_m/N_m \leq \lim_{m \to \infty} \frac{\altoverlapconst}{m} = 0$, so~\cref{eq:strong-uniform-density-asymptotic} holds.
    It remains to verify the large overlap condition.

    Let $v, w \in \superlevel{X}{a_0}^{(0)}$ satisfy $h(v)=m>a_0$, and assume that $w=v/a$ for some $a\in\lkdown(v)^{(0)}$.
    By~\localref{prop:full-descending-link-density-criterion:density}, all but at most $b_m+1$ descending edges at $v$ are distinct from and compatible with $a$.
    For every such edge $e$, one has $e/a\in\lkdown(w)^{(0)}$ and
    $\dualhyperplane{e}=\dualhyperplane{e/a}$ by~\cref{def:descending-edge-transport,lem:hyperplane-transport:transport}.
    By~\cref{lem:hyperplane-transport:injective}, these hyperplanes are pairwise distinct.
    Hence,
    \[
        \size{\densityhyperplanes{v}\cap\densityhyperplanes{w}}
        \geq N_m-b_m-1.
    \]
    Since~\cref{lem:hyperplane-transport:injective} also gives
    $\size{\densityhyperplanes{v}}=N_m$ and
    $\size{\densityhyperplanes{w}}=N_{m-1}$, we obtain
    \begin{align*}
        \size{\densityhyperplanes{v}\symdiff\densityhyperplanes{w}}
         &=N_m+N_{m-1}-2\size{\densityhyperplanes{v}\cap\densityhyperplanes{w}} \\
         &\leq N_{m-1}-N_m+2b_m+2 \\
         &\leq\abs{N_m-N_{m-1}}+2b_m+2 \\
         &\leq 4(b_m+\abs{N_m-N_{m-1}}).
    \end{align*}
    Here the last inequality uses $b_m\geq1$.
    Moreover, $\size{\densityhyperplanes{v}\cup\densityhyperplanes{w}}\geq N_m$, so
    \[
        \frac{\size{\densityhyperplanes{v}\symdiff\densityhyperplanes{w}}}
        {\size{\densityhyperplanes{v}\cup\densityhyperplanes{w}}}
        \leq \frac{4(b_m+\abs{N_m-N_{m-1}})}{N_m}
        \leq \frac{4\altoverlapconst}{m}.
    \]
    This proves the large overlap condition with $\overlapconstant = 4\altoverlapconst$.
\end{proof}

\begin{defn}%
    \label{def:lifted-density-contraction}
    Let $p\geq0$, let $v\in X^{(0)}$, let $b\in\bbN_0$, and let
    $E\subseteq\lkdown(v)^{(0)}$ be $b$\=/dense in $\lkdown(v)$ and $(R_p,b)$\=/good.
    For $e\in E$, let $F_{e,\ast}$ be the contraction of $\lkdown(v)$ from~\cref{def:density-contraction}.
    We define
    \[
        \linkcontraction_{v,e,k} \defq \linkshift_{v,k+1}\circ F_{e,k-1}\circ\linkshift_{v,k}^{-1} \colon\heightmodule{v}{k}\longrightarrow\heightmodule{v}{k+1}
    \]
    for $0\leq k\leq p+1$.
    We also set $\linkcontraction_{v,e,-1}=0$.
\end{defn}

\begin{lem}%
    \label{lem:lifted-density-contraction-bounds}
    Fix $p\geq0$ and suppose that the descending links satisfy strong uniform density above height $a_0$.
    There exist an integer $a_p>a_0$ and a constant $M_p\geq1$ with the following property.
    For every vertex $v$ of height $m\geq a_p$ and every $e\in \densityset{v}$, we have
    \begin{refenum}
        \item\label{lem:lifted-density-contraction-bounds:goodness}
              $\densityset{v}$ is $b_m$\=/dense in $\lkdown(v)$ and $(R_p,b_m)$\=/good;
        \item\label{lem:lifted-density-contraction-bounds:contraction}
              $\topboundary\linkcontraction_{v,e,k}+\linkcontraction_{v,e,k-1}\topboundary=\id$ on $\heightmodule{v}{k}$ for every $0\leq k\leq p+1$;
        \item\label{lem:lifted-density-contraction-bounds:norm}
              $\norm{\linkcontraction_{v,e,k}(Q)}\leq M_p$
              for every $0\leq k\leq p+1$
              and every basis cube $Q\in\heightmodule{v}{k}$;
        \item\label{lem:lifted-density-contraction-bounds:cone}
              $\linkcontraction_{v,e,k}(\cubecell{v}{\sigma}) = \cubecell{v}{e,\sigma}$
              for every $0\leq k\leq p+1$
              and every $(k-1)$\=/simplex $\sigma$ which is compatible with $e$.
    \end{refenum}
\end{lem}
\begin{proof}
    By~\cref{eq:strong-uniform-density-asymptotic},
    $\lim_{m \to \infty} b_m/N_m = 0$, so
    we can find an integer $a_p>a_0$ such that
    $R_pb_m<N_m$ whenever $m\geq a_p$.
    Let $v$ have height $m\geq a_p$.
    By~\cref{def:strong-uniform-density-descending-links:density}, the set $\densityset{v}$ is $b_m$\=/dense in $\lkdown(v)$ and $\size{\densityset{v}}\geq N_m>R_pb_m$.
    Thus, $\densityset{v}$ is $(R_p,b_m)$\=/good.

    Under the maps $\linkshift_{v,k}$,~\localref{lem:lifted-density-contraction-bounds:contraction} follows from~\cref{cor:quantitative-density-contraction:contraction}.
    Statements~\localref{lem:lifted-density-contraction-bounds:norm} and~\localref{lem:lifted-density-contraction-bounds:cone} follow from~\cref{cor:quantitative-density-contraction:basis-norm,cor:quantitative-density-contraction:cone-formula}.
    The constants $K_{-1},\ldots,K_p$ from~\cref{lem:quantitative-density-homotopy} are independent of $v$ and $e$, so $M_p \defq \max_{j} K_j + 1$ gives a suitable value of $M_p$.
\end{proof}

\begin{defn}[Descending choices and levelwise contractions]%
    \label{def:descending-choice}
    Suppose that the descending links satisfy strong uniform density above height $a_0$.
    Let $W \subseteq \superlevel{X}{a_0}^{(0)}$.
    A \emph{descending choice on $W$} is a family $\descChoice=(\descChoice(v))_{v\in W}$ with $\descChoice(v)\in\densityset{v}$ for every $v\in W$.
    Its \emph{associated hyperplane at $v$} is
    $\associatedhyperplane{\descChoice}{v}\defq\dualhyperplane{\descChoice(v)}$.

    Let $p\geq0$, and let $a_p>a_0$ be a height threshold supplied by~\cref{lem:lifted-density-contraction-bounds}.
    Suppose that $r\geq a_p$.
    For every $v\in W$ of height $r$, form the contractions $\linkcontraction_{v,\descChoice(v),k}$ using $\densityset{v}$.
    The \emph{levelwise contraction at height $r$ associated with $\descChoice$} is
    \[
        \linkcontraction^{\descChoice}_{r,k} \defq \;\bigoplus_{\mathclap{\substack{v\in W\\h(v)=r}}}\linkcontraction_{v,\descChoice(v),k} \colon \;\bigoplus_{\mathclap{\substack{v\in W\\h(v)=r}}} \heightmodule{v}{k} \longrightarrow\heightmodule{r}{k+1},
    \]
    for $0\leq k\leq p+1$.
\end{defn}

\begin{rem}%
    \label{rem:levelwise-contraction}
    For every descending choice $\descChoice$,~\cref{lem:lifted-density-contraction-bounds:contraction} gives
    \(
    \topboundary\linkcontraction^{\descChoice}_{r,k}
    +\linkcontraction^{\descChoice}_{r,k-1}\topboundary
    =\id
    \)
    on $\bigoplus_{v\in W,\,h(v)=r}\heightmodule{v}{k}$ for $0\leq k\leq p+1$.
\end{rem}

\begin{defn}%
    \label{def:one-level-pushing-operator}
    In the setting of~\cref{def:descending-choice}, the \emph{one-level pushing operator associated to $\descChoice$ at height $r$} is the collection of maps
    \[
        \morsepush^{\descChoice}_{r,k} \defq -\bottomboundary\linkcontraction^{\descChoice}_{r,k} \colon \;\bigoplus_{\mathclap{\substack{v\in W\\h(v)=r}}}\heightmodule{v}{k} \longrightarrow\heightmodule{r-1}{k},
    \]
    for $0\leq k\leq p+1$.
\end{defn}

\begin{rem}%
    \label{rem:pushing-construction}
    Fix $k\geq1$, take $p=k-1$ in the preceding construction, and suppose that $r>a_{k-1}$.
    The pushing operator $\morsepush^{\descChoice}_{r,k}$
    is a chain-level realization of pushing along the cones over descending links in cubical Morse theory (compare~\citelist{\cite{BestvinaBrady1997}*{Lemma~2.5}\cite{Bestvina2008PLMorse}*{Proposition~2.7}\cite{Genevois2023}*{Section~4.3}}).
    Let $b \in C_k(\sublevel{X}{r})$ and decompose $b$ as $b = x_{\leq r-1} + x_{r}$
    with $x_{\leq r-1} \in C_k(\sublevel{X}{r-1})$ and $x_{r} \in \heightmodule{r}{k}$.
    Suppose further that $\partial b \in C_{k-1}(\sublevel{X}{r-1})$.
    Then $\topboundary x_r=0$, so $\topboundary\linkcontraction^{\descChoice}_{r,k}x_r=x_r$
    by~\cref{rem:levelwise-contraction}.
    Hence, using $\partial=\topboundary+\bottomboundary$, we get
    \begin{align}%
        \label{eq:pushing-construction-new-chain}
        b-\partial\linkcontraction^{\descChoice}_{r,k}x_r
        =(b-x_r)-\bottomboundary\linkcontraction^{\descChoice}_{r,k}x_r
        =(b-x_r)+\morsepush^{\descChoice}_{r,k}x_r \in C_k(\sublevel{X}{r-1}).
    \end{align}
    Thus, subtracting the boundary of $\linkcontraction^{\descChoice}_{r,k}x_r$ removes the component $x_r$ at height $r$ and replaces it by $t_{r-1} \defq \morsepush^{\descChoice}_{r,k}x_r$ at height $r-1$.
    Denote the resulting chain by $b'$.
    It satisfies
    \begin{align*}%
        \partial b' = \partial(b-\partial\linkcontraction^{\descChoice}_{r,k}x_r)
        = \partial b.
    \end{align*}

    By~\cref{lem:cubical-boundary-splitting:norm,lem:lifted-density-contraction-bounds:norm}, we obtain
    \(
    \norm{\morsepush^{\descChoice}_{r,k}x_r} \leq (k+1)M_{k-1}\norm{x_r}
    \),
    so~\cref{eq:pushing-construction-new-chain} yields
    \[
        \norm{b'}
        \leq \norm{b} + (k+1)M_{k-1} \norm{x_r}
        \leq ((k+1)M_{k-1} + 1)\norm{b}.
    \]

    Let $x_{r-1}$ be the component of $b$ at height $r-1$,
    that is, $x_{\leq r-1} = x_{\leq r-2} + x_{r-1}$ with $x_{\leq r-2} \in C_k(\sublevel{X}{r-2})$ and $x_{r-1} \in \heightmodule{r-1}{k}$.
    If, in addition to $\partial b \in C_{k-1}(\sublevel{X}{r-1})$, we know
    that $\partial b = \partial b' \in C_{k-1}(\sublevel{X}{r-2})$, we can repeat the above
    height reduction with $b'$.
    This time, $x_{r-1} + t_{r-1}$ is replaced by
    $t_{r-2} \defq \morsepush^{\descChoice}_{r-1,k}(x_{r-1} + t_{r-1})$
    to obtain a chain $b''$ of maximal height $r-2$.
    Let $E_k \defq (k+1)M_{k-1} + 1$. Then
    \begin{align*}
        \norm{b''} \leq E_k\norm{b'} \leq E_k^2 \norm{b}.
    \end{align*}
    One sees that repeated pushing will, in general, cause the norm to grow
    exponentially with the height of the chain.
\end{rem}

\subsection{Probabilistic pushing}%
\label{sec:probabilistic-pushing}
For the remainder of this~\lcnamecref{sec:probabilistic-pushing}, suppose that the descending links satisfy strong uniform density above height $a_0$.

Let $c \in C_{k-1}(\sublevel{X}{a})$ be a cycle and $b \in C_k(X)$ a filling of $c$.
As explained in~\cref{rem:pushing-construction}, naive reduction of the height of $b$ by one might increase the norm by a constant multiplicative factor $E_k > 0$.
Performing this iteratively to push the filling into $\sublevel{X}{a}$
would therefore only produce a filling with norm at most exponential
in the height of $b$, that is, if $H\geq a$ and $b \in C_k(\sublevel{X}{H})$,
then the norm of the pushed-down filling is bounded by $E_k^{H-a}\norm{b}$.

The goal of this~\lcnamecref{sec:probabilistic-pushing} is to obtain an
improved bound for reducing the height of a filling by one.
Write $b^{(H)} \defq b = x_{\leq a} + \sum_{m=a+1}^H x_m$, where $x_{\leq a} \in C_k(\sublevel{X}{a})$ and $x_m \in \heightmodule{m}{k}$.
Each height reduction step applies the pushing map from~\cref{rem:pushing-construction} to obtain a new chain with height decreased by one.
After $r$ reduction steps, the filling will be of the form
\[
    b^{(H-r)} = x_{\leq a} + \cdots + x_{H-r-1} + x_{H-r} + t_{H-r},
\]
where $t_{H-r}$ is the image of $x_{H-r+1} + t_{H-r+1}$
under $\morsepush^{\descChoice}_{H-r+1,k} = -\bottomboundary\linkcontraction^{\descChoice}_{H-r+1,k}$.
In particular, the pushing map that replaces $b^{(H-r)}$ by $b^{(H-r-1)}$ modifies only the cells of maximal height,
that is, the term $y_{H-r} = x_{H-r} + t_{H-r}$.
Note that $t_{H-r}$ lies in the image of $\bottomboundary$ and hence is a $\bottomboundary$\=/cycle.
In order to obtain polynomial bounds for the norm of a filling which has been pushed into $\sublevel{X}{a}$, we show that $\descChoice$ can be chosen such
that the increase in norm when pushing the $t_{H-r}$-components is much
smaller than that for the $x_{H-r}$-components.
More precisely, we prove that for each fixed degree $k$ there are $A_k \geq 0$ and $a_k \in \bbN$ such that for each $m > a_k$ and every $\bottomboundary$\=/cycle
$t \in \heightmodule{m}{k}$ there exists a descending choice $\descChoice$ satisfying
\begin{equation}%
    \label{eq:goal-estimate}
    \norm{\morsepush^{\descChoice}_{m,k}t} \leq \left(1+\frac{A_k}{m}\right)\norm{t}.
\end{equation}
Because each $x_m$ is only pushed once, the general multiplicative cost for reducing
height by one only contributes $E_k \norm{b}$ to the norm of the pushed-down filling.

We first describe which phenomena can occur when pushing a $\bottomboundary$\=/cycle.
To illustrate this, consider a $3$-dimensional cube $Q=\cubecell{v}{a,b,c}$ with $h(v)=m+1$, and put
$t=\bottomboundary Q\in\heightmodule{m}{2}$,
that is, $t = -A + B - C$, where
$A=\cubecell{v/a}{b/a,c/a}$, $B=\cubecell{v/b}{a/b,c/b}$, and $C=\cubecell{v/c}{a/c,b/c}$.
Thus, $\bottomboundary t=0$, and we need to estimate the norm of
\[
    \morsepush^{\descChoice}_{m,2}t=-\bottomboundary\linkcontraction^{\descChoice}_{m,2}t.
\]
The situation is depicted in~\cref{fig:pushing}.

Put $x=b/a$, $y=c/a$, and $e=\descChoice(v/a)$.
We distinguish the cases when $e$ is incompatible with $x$ or $y$
and when $e$ is compatible with both $x$ and $y$.

If $e$ is incompatible with $x$ or $y$, the norm of $\morsepush^{\descChoice}_{m,2}t$ generally only satisfies the bound given by~\cref{lem:lifted-density-contraction-bounds}. The following~\lcnamecref{rem:incompatible-pushing} illustrates an explicit filling of this kind.

\begin{rem}%
    \label{rem:incompatible-pushing}
    In $\lkdown(v/a)$, the chains $F_{e,0}[x]$ and $F_{e,0}[y]$ join $e$ to the two endpoints of $[x,y]$, since $\partial F_{e,0}[x]=[x]-[e]$ and
    $\partial F_{e,0}[y]=[y]-[e]$.
    Together with $[x,y]$, they form the cycle
    \[
        [x,y]-F_{e,0}[y]+F_{e,0}[x]
        =\partial F_{e,1}[x,y].
    \]
    This cycle is first filled in the descending link, then lifted to $X$ using $\linkshift_{v/a,3}$, and finally pushed down using $\bottomboundary$.
\end{rem}

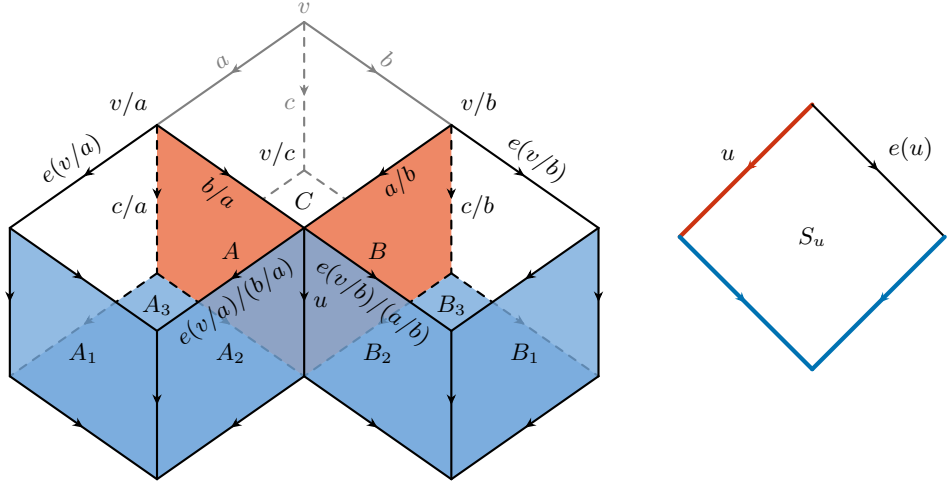
\begin{figure}
    \centering
    \pgfmathsetmacro{\CurrentUpAxisAngle}{54.735610317} %
\tikzset{%
    rotate around current up/.style={
            rotate around z=45,
            rotate around y=\CurrentUpAxisAngle,
            rotate around z=#1,
            rotate around y=-\CurrentUpAxisAngle,
            rotate around z=-45,
        }
}%
\def\SquareSide{2.75}%
\begin{tikzpicture}[
        baseline=(current bounding box.center),
        3d view={-45}{65},
        line join=round,
        line cap=round,
        font=\small,
        rotate around current up=60
    ]

    \def\LX{\SquareSide}
    \def\LY{\SquareSide}
    \def\LZ{\SquareSide}

    \cubeCoordinates
    \cubeCoordinates[prefix=a,z=-\LZ]
    \cubeCoordinates[prefix=b,y=-\LY]

    \begin{scope}[gray,thick,dashed,every path/.append style={descending edge}]
        \draw (C)--(BC);
        \draw (V)--(C);
        \draw (C)--(AC);
    \end{scope}

    \begin{scope}[light Red,opacity=1]
        \fill (O)--(AB)--(B)--(BC)--cycle;
        \fill (O)--(AB)--(A)--(AC)--cycle;
    \end{scope}

    \begin{scope}[thick,every path/.append style={descending edge}]
        \draw (A)--(AB);
        \draw (B)--(AB);
        \draw[gray] (V)--(A);
        \draw[gray] (V)--(B);

        \begin{scope}[dashed]
            \draw (A)--(AC);
            \draw (B)--(BC);
        \end{scope}
    \end{scope}

    \begin{scope}[light Blue,opacity=0.8]
        \fill (aC)--(aAC)--(aO)--(aBC)--cycle;
    \end{scope}

    \begin{scope}[thick,solid,every path/.append style={descending edge}]
        \begin{scope}[dashed]
            \draw (aC)--(aAC);
            \draw (aC)--(aBC);
        \end{scope}
    \end{scope}

    \begin{scope}[light Blue,opacity=0.8]
        \fill (aO)--(aAB)--(aB)--(aBC)--cycle;
        \fill (aA)--(aAB)--(aO)--(aAC)--cycle;
    \end{scope}

    \begin{scope}[thick,solid,every path/.append style={descending edge}]
        \draw (aBC)--(aO);
        \draw (aAB)--(aO);
        \draw (aA)--(aAB);
        \draw (aA)--(aAC);
        \draw (aAC)--(aO);
        \draw (aB)--(aAB);
        \draw (aV)--(aA);
    \end{scope}

    \begin{scope}[light Blue,opacity=0.8]
        \fill (bC)--(bAC)--(bO)--(bBC)--cycle;
    \end{scope}

    \begin{scope}[thick,solid,every path/.append style={descending edge}]
        \begin{scope}[dashed]
            \draw (bC)--(bBC);
            \draw (bC)--(bAC);
        \end{scope}
    \end{scope}

    \begin{scope}[light Blue,opacity=0.8]
        \fill (bO)--(bAB)--(bA)--(bAC)--cycle;
        \fill (bAB)--(bB)--(bBC)--(bO)--cycle;
    \end{scope}

    \begin{scope}[thick,solid,every path/.append style={descending edge}]
        \draw (bBC)--(bO);
        \draw (bB)--(bBC);
        \draw (bAB)--(bO);
        \draw (bA)--(bAB);
        \draw (bAC)--(bO);
        \draw (bB)--(bAB);
        \draw (bV)--(bB);
    \end{scope}

    \begin{scope}[thick,solid,every path/.append style={descending edge}]
        \draw (AB)--(O);
    \end{scope}

    \node[above left,yshift=-2pt] at (C) {$v/c$};
    \node[above right] at (B) {$v/b$};
    \node[above left] at (A) {$v/a$};

    \begin{scope}[gray]
        \node[above] at (V) {$v$};
        \node[left] at ($(V)!0.55!(C)$) {$c$};
        \path[draw=none] (V) -- node[above,sloped] {$b$} (B);
        \path[draw=none] (V) -- node[above,sloped] {$a$} (A);
    \end{scope}

    \path[draw=none] (A) -- node[below,sloped] {$b/a$} (AB);
    \node[left] at ($(A)!0.55!(AC)$) {$c/a$};

    \path[draw=none] (B) -- node[pos=0.55,below right=-2pt,sloped] {$a/b$} (AB);
    \node[right] at ($(B)!0.55!(BC)$) {$c/b$};

    \node at ($(AC)!0.5!(AB)$) {$A$};
    \node at ($(BC)!0.5!(AB)$) {$B$};
    \node[below=5pt] at (C) {$C$};
    \node[right] at ($(AB)!0.5!(O)$) {$u$};

    \path[draw=none] (aV) -- node[above,sloped] {$e(v/a)$} (aA);
    \path[draw=none] (bV) -- node[above,sloped] {$e(v/b)$} (bB);

    \path[draw=none] (aAB) -- node[below, sloped,pos=0.45] {$e(v/a)/(b/a)$} (aB);
    \path[draw=none] (bA) -- node[below, sloped,pos=0.55] {$e(v/b)/(a/b)$} (bAB);

    \node at ($(aAC)!0.5!(aAB)$) {$A_1$};
    \node at ($(aBC)!0.5!(aAB)$) {$A_2$};
    \node[below=5pt] at (aC) {$A_3$};

    \node at ($(bAC)!0.5!(bAB)$) {$B_2$};
    \node at ($(bBC)!0.5!(bAB)$) {$B_1$};
    \node[below=5pt] at (bC) {$B_3$};
\end{tikzpicture}
\hspace{0.75cm}
\begin{tikzpicture}[
        baseline=(current bounding box.center),
        line join=round,
        line cap=round,
        font=\small,
        rotate=45,
        scale=0.9
    ]

    \begin{scope}[thick,every path/.append style={descending edge}]
        \draw[ultra thick, Red] (\SquareSide,\SquareSide) -- node[black,above left] {$u$} (0,\SquareSide);
        \draw (\SquareSide,\SquareSide) -- node[above right] {$e(u)$} (\SquareSide,0);
        \draw[ultra thick, Blue] (0,\SquareSide) -- (0,0);
        \draw[ultra thick, Blue] (\SquareSide,0) -- (0,0);
    \end{scope}
    \node at ({0.5*\SquareSide},{0.5*\SquareSide}) {$S_u$};
\end{tikzpicture}

    \caption{%
        \label{fig:pushing}
        Pushing two adjacent faces in the lower boundary of a cube.
        The pushing behaves \enquote{nicely} if the faces $A_2$ and $B_2$ coincide.
    }
\end{figure}

If $e$ is compatible with $[x,y]$ and equals $x$ or $y$,
the pushing operation is degenerate, and one gets $\morsepush^{\descChoice}_{m,2} A = 0$.

Finally, assume $e$ is compatible with $[x,y]$ and is not equal to either $x$ or $y$.
Then $[e,x,y]$ is a $2$\=/simplex in $\lkdown(v/a)$, and one obtains
\begin{align*}
    \morsepush^{\descChoice}_{m,2}A %
     &=A_1-A_2+A_3,
\end{align*}
where $A_1$, $A_2$, and $A_3$ are the bottom faces of the cube
$\cubecell{v/a}{e,b/a,c/a}$.
We now also assume that the choices $\descChoice(v/b)$ and $\descChoice(v/c)$
satisfy the respective compatibility conditions, so that
\begin{align*}
    \morsepush^{\descChoice}_{m,2}B %
     &=B_1-B_2+B_3, \\
    \morsepush^{\descChoice}_{m,2}C %
     &=C_1-C_2+C_3.
\end{align*}
In total, we obtain
\begin{align*}
    \morsepush^{\descChoice}_{m,2} t
     &= \morsepush^{\descChoice}_{m,2} (-A + B - C) \\
     &= -(A_1 - A_2 + A_3) + (B_1 - B_2 + B_3) - (C_1 - C_2 + C_3)
    \\
     &= \underbrace{(-A_1 + B_1 - C_1)}_{\qdef \principalpart{t}{\descChoice}} + \underbrace{(A_2 - B_2) + (B_3 - C_3) + (C_2 - A_3)}_{\qdef \residuepart{t}{\descChoice}}.
\end{align*}
While this can a priori increase the norm by a factor of $3$,
a suitable choice of $\descChoice$ can ensure that in most cases $\residuepart{t}{\descChoice} = 0$ and thus
$\norm{\morsepush^{\descChoice}_{m,2} t} = \norm{t}$.
To explain this, we focus on the two faces $A$ and $B$, which meet in the edge $u$.
The other pairs of faces can be treated in the same way.
Denote by $\descChoice(u)$ the choice at the maximal vertex of $u$ and by $S_u$
the square obtained by pushing $u$ along $\descChoice(u)$, shown on the right in~\cref{fig:pushing}.
By~\cref{lem:hyperplane-transport}, we have the equivalence
\[
    A_2=S_u=B_2
    \quad\Longleftrightarrow\quad
    \dualhyperplane{\descChoice(v/a)}
    =\dualhyperplane{\descChoice(u)}
    =\dualhyperplane{\descChoice(v/b)}.
\]
Thus, the terms $A_2$ and $B_2$ cancel in $\morsepush^{\descChoice}_{m,2}t$ precisely if
$\associatedhyperplane{\descChoice}{v/a} = \associatedhyperplane{\descChoice}{u}$
and
$\associatedhyperplane{\descChoice}{v/b} = \associatedhyperplane{\descChoice}{u}$.

If these cancellation conditions are satisfied for every pair of faces,
we indeed have $\morsepush^{\descChoice}_{m,2} t = \principalpart{t}{\descChoice}$
and thus $\norm{\morsepush^{\descChoice}_{m,2} t} = \norm{t}$ (see~\cref{fig:pushing-cancellation}).
\begin{figure}
    \centering
    \begin{subfigure}[b]{0.6\textwidth}
        \centering
        \pgfmathsetmacro{\CurrentUpAxisAngle}{54.735610317} %
\tikzset{%
    rotate around current up/.style={
            rotate around z=45,
            rotate around y=\CurrentUpAxisAngle,
            rotate around z=#1,
            rotate around y=-\CurrentUpAxisAngle,
            rotate around z=-45,
        }
}%
\begin{tikzpicture}[
        baseline=(current bounding box.south),
        3d view={-45}{60},
        line join=round,
        line cap=round,
        font=\small,
        rotate around current up=55
    ]
    \def\SquareSide{2.6}%
    \def\LightRedFaceOpacity{0.4}%
    \def\LightBlueFaceOpacity{0.8}%

    \def\LX{\SquareSide}
    \def\LY{\SquareSide}
    \def\LZ{\SquareSide}

    \cubeCoordinates
    \cubeCoordinates[prefix=a,z=-\LZ]
    \cubeCoordinates[prefix=b,y=-\LY]
    \cubeCoordinates[prefix=c,x=-\LX]

    \begin{scope}[light Red,opacity=\LightRedFaceOpacity]
        \fill (cV)--(cB)--(cAB)--(cA)--cycle;
    \end{scope}

    \begin{scope}[gray,thick,solid]
        \draw[descending edge] (cV)--(cA);
        \draw[descending edge=0.4] (cV)--(cB);
        \draw[dashed,descending edge=0.4] (cV)--(cC);
    \end{scope}

    \begin{scope}[light Blue,opacity=\LightBlueFaceOpacity]
        \fill (cC)--(cAC)--(cO)--(cBC)--cycle;
    \end{scope}

    \begin{scope}[thick,dashed]
        \draw (cC)--(cAC);
        \draw (cC)--(cBC);
    \end{scope}

    \begin{scope}[light Blue,opacity=\LightBlueFaceOpacity]
        \fill (cO)--(cAB)--(cB)--(cBC)--cycle;
        \fill (cA)--(cAB)--(cO)--(cAC)--cycle;
    \end{scope}

    \begin{scope}[thick,solid,every path/.append style={descending edge}]
        \draw (cBC)--(cO);
        \draw (cAB)--(cO);
        \draw (cAC)--(cO);
    \end{scope}

    \begin{scope}[thick,dashed,every path/.append style={}]
        \draw (cA)--(cAC);
        \draw (cB)--(cBC);
    \end{scope}

    \begin{scope}[light Red,opacity=\LightRedFaceOpacity]
        \fill (aV)--(aB)--(aBC)--(aC)--cycle;
    \end{scope}

    \begin{scope}[gray,thick,solid,every path/.append style={descending edge}]
        \draw (aV)--(aA);
        \draw (aV)--(aB);
        \draw (aV)--(aC);
    \end{scope}

    \begin{scope}[light Blue,opacity=\LightBlueFaceOpacity]
        \fill (aC)--(aAC)--(aO)--(aBC)--cycle;
    \end{scope}

    \begin{scope}[thick,solid,every path/.append style={descending edge}]
        \begin{scope}[dashed]
            \draw (aC)--(aAC);
            \draw (aC)--(aBC);
        \end{scope}
    \end{scope}

    \begin{scope}[light Blue,opacity=\LightBlueFaceOpacity]
        \fill (aO)--(aAB)--(aB)--(aBC)--cycle;
        \fill (aA)--(aAB)--(aO)--(aAC)--cycle;
    \end{scope}

    \begin{scope}[thick,solid,every path/.append style={descending edge}]
        \begin{scope}[gray]
            \draw (aV)--(aA);
            \draw (aV)--(aB);
            \draw (aV)--(aC);
        \end{scope}
        \draw (aBC)--(aO);
        \draw (aB)--(aBC);
        \draw (aAB)--(aO);
        \draw (aA)--(aAB);
        \draw (aA)--(aAC);
        \draw (aAC)--(aO);
        \draw (aB)--(aAB);
    \end{scope}

    \begin{scope}[light Red,opacity=\LightRedFaceOpacity]
        \fill (bV)--(bA)--(bAC)--(bC)--cycle;
    \end{scope}

    \begin{scope}[gray,thick,solid,every path/.append style={descending edge}]
        \draw (bV)--(bA);
        \draw (bV)--(bB);
        \draw (bV)--(bC);
    \end{scope}

    \begin{scope}[light Blue,opacity=\LightBlueFaceOpacity]
        \fill (bC)--(bAC)--(bO)--(bBC)--cycle;
    \end{scope}

    \begin{scope}[thick,solid,every path/.append style={descending edge}]
        \begin{scope}[dashed]
            \draw (bC)--(bBC);
            \draw (bC)--(bAC);
        \end{scope}
    \end{scope}

    \begin{scope}[light Blue,opacity=\LightBlueFaceOpacity]
        \fill (bO)--(bAB)--(bA)--(bAC)--cycle;
        \fill (bAB)--(bB)--(bBC)--(bO)--cycle;
    \end{scope}

    \begin{scope}[thick,solid,every path/.append style={descending edge}]
        \draw (bBC)--(bO);
        \draw (bB)--(bBC);
        \draw (bAB)--(bO);
        \draw (bA)--(bAB);
        \draw (bA)--(bAC);
        \draw (bAC)--(bO);
        \draw (bB)--(bAB);
    \end{scope}

    \begin{scope}[thick,solid,every path/.append style={descending edge}]
        \draw (AB)--(O);
    \end{scope}

    \node at ($(aAC)!0.5!(aAB)$) {$A_1$};
    \node at ($(aBC)!0.5!(aAB)$) {$A_2$};
    \node[below=5pt] at (aC) {$A_3$};

    \node at ($(bAC)!0.5!(bAB)$) {$B_2$};
    \node at ($(bBC)!0.5!(bAB)$) {$B_1$};
    \node[below=5pt] at (bC) {$B_3$};

    \node[below=5pt] at (cC) {$C_1$};
    \node at ($(cA)!0.75!(cO)$) {$C_2$};
    \node at ($(cB)!0.75!(cO)$) {$C_3$};

    \node at ($(aV)!0.5!(aBC)$) {$A$};
    \node at ($(bV)!0.5!(bAC)$) {$B$};
    \node[below,yshift=-8pt,xshift=8pt] at (cV) {$C$};

\end{tikzpicture}

        \caption{%
            \label{fig:pushing-cancellation-none}
            No cancellation.}
    \end{subfigure}%
    \hfill%
    \begin{subfigure}[b]{0.35\textwidth}
        \centering
        \begin{tikzpicture}[
        baseline=(current bounding box.south),
        3d view={-10}{-0},
        line join=round,
        line cap=round,
        font=\small,
    ]
    \def\SquareSide{3}%
    \def\LightRedFaceOpacity{0.4}%
    \def\LightBlueFaceOpacity{0.8}%

    \pgfmathsetmacro{\HexRadius}{0.72*\SquareSide}
    \pgfmathsetmacro{\HexX}{0.8660254*\HexRadius}
    \pgfmathsetmacro{\HexY}{0.5*\HexRadius}
    \pgfmathsetmacro{\CubeHeight}{0.72*\SquareSide}
    \pgfmathsetmacro{\OuterEdgeShift}{0.37*\SquareSide}

    \path (tpp cs:x=0,y=0,z=-\OuterEdgeShift) coordinate (O);
    \path (tpp cs:x=0,y=0,z={\CubeHeight-\OuterEdgeShift}) coordinate (AB);

    \path (tpp cs:x=0,y=0,z=-\OuterEdgeShift) coordinate (aO);
    \path (tpp cs:x=0,y=0,z={\CubeHeight-\OuterEdgeShift}) coordinate (aBC);
    \path (tpp cs:x=-\HexX,y=\HexY,z=0) coordinate (aAC);
    \path (tpp cs:x=0,y=-\HexRadius,z=0) coordinate (aAB);
    \path (tpp cs:x=-\HexX,y=\HexY,z=\CubeHeight) coordinate (aC);
    \path (tpp cs:x=0,y=-\HexRadius,z=\CubeHeight) coordinate (aB);
    \path (tpp cs:x=-\HexX,y=-\HexY,z=\OuterEdgeShift) coordinate (aA);
    \path (tpp cs:x=-\HexX,y=-\HexY,z={\CubeHeight+\OuterEdgeShift}) coordinate (aV);

    \path (tpp cs:x=0,y=0,z=-\OuterEdgeShift) coordinate (bO);
    \path (tpp cs:x=\HexX,y=\HexY,z=0) coordinate (bBC);
    \path (tpp cs:x=0,y=0,z={\CubeHeight-\OuterEdgeShift}) coordinate (bAC);
    \path (tpp cs:x=0,y=-\HexRadius,z=0) coordinate (bAB);
    \path (tpp cs:x=\HexX,y=\HexY,z=\CubeHeight) coordinate (bC);
    \path (tpp cs:x=\HexX,y=-\HexY,z=\OuterEdgeShift) coordinate (bB);
    \path (tpp cs:x=0,y=-\HexRadius,z=\CubeHeight) coordinate (bA);
    \path (tpp cs:x=\HexX,y=-\HexY,z={\CubeHeight+\OuterEdgeShift}) coordinate (bV);

    \path (tpp cs:x=0,y=0,z=-\OuterEdgeShift) coordinate (cO);
    \path (tpp cs:x=\HexX,y=\HexY,z=0) coordinate (cBC);
    \path (tpp cs:x=-\HexX,y=\HexY,z=0) coordinate (cAC);
    \path (tpp cs:x=0,y=0,z={\CubeHeight-\OuterEdgeShift}) coordinate (cAB);
    \path (tpp cs:x=0,y=\HexRadius,z=\OuterEdgeShift) coordinate (cC);
    \path (tpp cs:x=\HexX,y=\HexY,z=\CubeHeight) coordinate (cB);
    \path (tpp cs:x=-\HexX,y=\HexY,z=\CubeHeight) coordinate (cA);
    \path (tpp cs:x=0,y=\HexRadius,z={\CubeHeight+\OuterEdgeShift}) coordinate (cV);

    \coordinate (cVaVIntersection) at
    (intersection of cV--cA and aV--aB);
    \coordinate (cVbVIntersection) at
    (intersection of cV--cB and bV--bA);
    \coordinate (cVCbVIntersection) at
    (intersection of cV--cC and bV--bA);

    \begin{scope}[gray,thick]
        \draw[dashed,descending edge] (cV)--(cC);
        \draw[solid] (cV)--(cVCbVIntersection);
    \end{scope}

    \begin{scope}[light Red,opacity=\LightRedFaceOpacity]
        \fill (cV)--(cB)--(AB)--(cA)--cycle;
    \end{scope}

    \begin{scope}[light Blue,opacity=\LightBlueFaceOpacity]
        \fill (cC)--(cAC)--(O)--(cBC)--cycle;
    \end{scope}

    \begin{scope}[thick,dashed,every path/.append style={descending edge}]
        \draw (cC)--(cAC);
        \draw (cC)--(cBC);
    \end{scope}

    \begin{scope}[gray,thick,dashed,every path/.append style={descending edge}]
        \draw (aC)--(aAC);
    \end{scope}

    \begin{scope}[gray,thick,dashed,every path/.append style={descending edge=0.4}]
        \draw (AB)--(O);
    \end{scope}

    \begin{scope}[light Blue,opacity=\LightBlueFaceOpacity]
        \fill (bB)--(bBC)--(O)--(bAB)--cycle;
    \end{scope}

    \begin{scope}[light Blue,opacity=\LightBlueFaceOpacity]
        \fill (aA)--(aAB)--(O)--(aAC)--cycle;
    \end{scope}

    \begin{scope}[dashed,gray,thick,every path/.append style={descending edge}]
        \draw (cV)--(cA);
        \draw (cV)--(cB);
    \end{scope}

    \begin{scope}[gray,thick,solid]
        \draw (cV)--(cVaVIntersection);
        \draw (cV)--(cVbVIntersection);
    \end{scope}

    \begin{scope}[light Red,opacity=\LightRedFaceOpacity]
        \fill (aV)--(aB)--(AB)--(aC)--cycle;
        \fill (bV)--(bA)--(AB)--(bC)--cycle;
    \end{scope}

    \begin{scope}[gray,thick,solid,every path/.append style={descending edge}]
        \draw (aV)--(aB);
        \draw (bV)--(bA);
        \draw (bV)--(bC);
        \begin{scope}[dashed]
            \draw (bC)--(AB);
            \draw (aV)--(aC);
            \draw (aB)--(AB);
            \draw (aC)--(AB);
        \end{scope}
    \end{scope}

    \begin{scope}[gray,thick,every path/.append style={descending edge}]
        \draw (aV)--(aA);
        \draw (bC)--(bBC);
        \draw (aB)--(aAB);
        \draw (bV)--(bB);
    \end{scope}

    \begin{scope}[thick,solid,every path/.append style={descending edge}]
        \draw (aAB)--(O);
        \draw (aAC)--(O);
        \draw (bBC)--(O);
        \draw (aA)--(aAB);
        \draw (aA)--(aAC);
        \draw (bB)--(bAB);
        \draw (bB)--(bBC);
    \end{scope}

    \node at ($(aAC)!0.5!(aAB)$) {$A_1$};
    \node at ($(bBC)!0.5!(bAB)$) {$B_1$};
    \node at ($(cC)!0.15!(cO)$) {$C_1$};

    \node at ($(aV)!0.5!(AB)$) {$A$};
    \node at ($(bV)!0.5!(AB)$) {$B$};
    \node[below,yshift=-8pt,xshift=8pt] at (cV) {$C$};
\end{tikzpicture}

        \caption{%
            \label{fig:pushing-cancellation-all}
            All possible cancellations.
        }
    \end{subfigure}
    \caption{%
        \label{fig:pushing-cancellation}
        The two extreme cases of cancellation
        when pushing the bottom boundary of a $3$\=/cube.
        In both cases, the original chain is shown
        in red and the pushed chain in blue.
    }
\end{figure}
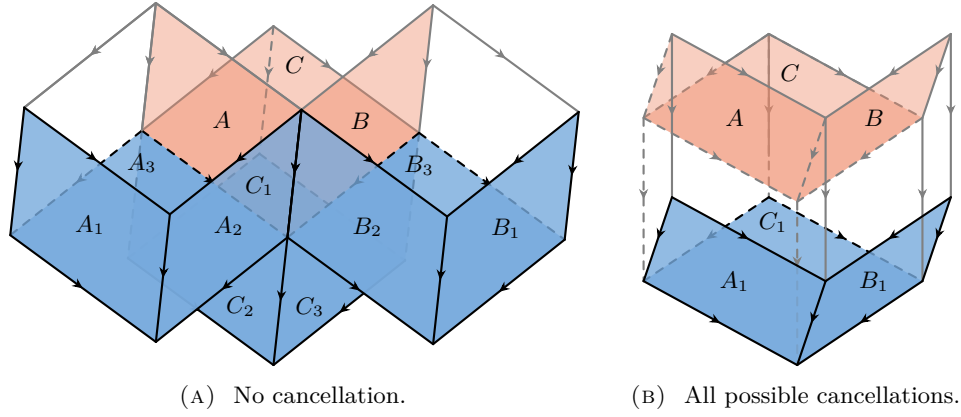

Going back to the general setting, this motivates the following~\lcnamecref{def:descending-choice-defect}.

\begin{defn}%
    \label{def:descending-choice-defect}
    Let $\descChoice$ be a descending choice on $W\subseteq \superlevel{X}{a_0}^{(0)}$, and let $Q=\cubecell{v}{\sigma}$ be a basis cube such that $v\in W$ and $v/a\in W$ for every $a\in \vertices{\sigma}$.
    The \emph{$\descChoice$\=/defect set} of $Q$ is
    \[
        \defectset{Q}{\descChoice} \defq \{a\in \vertices{\sigma}\mid \associatedhyperplane{\descChoice}{v}\neq\associatedhyperplane{\descChoice}{v/a}\}.
    \]
    Thus, $a\notin\defectset{Q}{\descChoice}$ precisely when the descending choices at the maximal vertex of $Q$ and at the opposite vertex of the edge $a$ have the same associated hyperplane.

    For a chain $t=\sum_Qt_QQ\in\heightmodule{m}{k}$, set
    \begin{align*}
        \uppervertices{t} &\defq \{ v \mid \cubecell{v}{\sigma} \in \supp(t) \} \\
        \lowervertices{t} &\defq \{ v/a \mid a \in \vertices{\sigma} \text{ for } \cubecell{v}{\sigma} \in \supp(t) \}
    \end{align*}
    and $\choicevertices{t} \defq \uppervertices{t} \cup \lowervertices{t}$.
    Assume that $\choicevertices{t} \subseteq W$.
    The \emph{$\descChoice$\=/defect} of $t$ is
    \[
        \defectweight{t}{\descChoice} \defq \;\sum_{\mathclap{Q\in\supp(t)}}\; \abs{t_Q} \size{\defectset{Q}{\descChoice}}.
    \]
\end{defn}

\begin{defn}%
    \label{def:cubical-good-bad-parts}
    Let $t=\sum_Q t_Q Q \in \heightmodule{m}{k}$, and let $\descChoice$ be a descending choice on a set containing $\choicevertices{t}$.
    A basis cube $Q=\cubecell{v}{\sigma}$ in $\supp(t)$ is \emph{$\descChoice$\=/good} if $\descChoice(v)$ is compatible with $\sigma$ and does not belong to $\vertices{\sigma}$, and it is \emph{$\descChoice$\=/bad} otherwise.
    Define
    \[
        \goodpart{t}{\descChoice}\defq\sum_{\mathclap{Q\text{ $\descChoice$\=/good}}}t_QQ,
        \quad
        \badpart{t}{\descChoice}\defq\sum_{\mathclap{Q\text{ $\descChoice$\=/bad}}}t_QQ.
    \]
\end{defn}

\begin{rem}%
    \label{rem:good-bad-part-estimate}
    Put $g=\goodpart{t}{\descChoice}$ and $b=\badpart{t}{\descChoice}$.
    Their supports partition $\supp(t)$, so $t = g + b$ and $\defectweight{t}{\descChoice}=\defectweight{g}{\descChoice}+\defectweight{b}{\descChoice}$.
    If $Q=\cubecell{v}{\sigma}$ is $\descChoice$\=/bad, there is an $a\in \vertices{\sigma}$ such that $\descChoice(v)$ is either incompatible with or equal to $a$.
    In either case,~\cref{lem:hyperplane-transport:transport} gives $a\in\defectset{Q}{\descChoice}$.
    Hence, $\size{\defectset{Q}{\descChoice}} \geq 1$ and
    in particular $\norm{b}\leq\defectweight{b}{\descChoice}\leq\defectweight{t}{\descChoice}$.
\end{rem}

\begin{lem}%
    \label{lem:deterministic-pushing-estimate}
    Let $k\geq1$.
    For $p=k-1$, let $a_k$ and $M_{k-1}\geq1$ be the constants obtained from~\cref{lem:lifted-density-contraction-bounds}.
    Let $m>a_k$, let $t\in\heightmodule{m}{k}$ satisfy $\bottomboundary t=0$, and let $\descChoice$ be a descending choice on a set containing $\choicevertices{t}$.
    Then
    \[
        \norm{\morsepush^{\descChoice}_{m,k}t} \leq \norm{t} + (2k + 1) M_{k-1}\defectweight{t}{\descChoice}\mpct{.}
    \]
\end{lem}
\begin{proof}
    For brevity, define $C\defq\bottomboundary\linkcontraction^{\descChoice}_{m,k}+\linkcontraction^{\descChoice}_{m-1,k-1}\bottomboundary$ and $M = M_{k-1}$.

    As before, let $g=\goodpart{t}{\descChoice}$ and $b=\badpart{t}{\descChoice}$.

    Let $Q=\cubecell{v}{\sigma}$ be a basis cube in $\supp(g)$,
    where $\sigma=[a_1,\ldots,a_k]$.
    Put $w_i=v/a_i$ and $\sigma_i=(\sigma\setminus a_i)/a_i$.
    The definitions of $\bottomboundary$ and $\linkcontraction^{\descChoice}_{m-1,k-1}$ give
    \begin{align}%
        \label{eq:det-push-b}
        \linkcontraction^{\descChoice}_{m-1,k-1}\bottomboundary Q &=\sum_{i=1}^k(-1)^i \linkcontraction_{w_i,\descChoice(w_i),k-1} \cubecell{w_i}{\sigma_i}.
    \end{align}
    By~\cref{lem:lifted-density-contraction-bounds:cone}, we have
    $\linkcontraction^{\descChoice}_{m,k}Q=\cubecell{v}{\descChoice(v),\sigma}$.
    The definition of $\bottomboundary$ then gives
    \begin{align}%
        \label{eq:det-push-a}
        \bottomboundary\linkcontraction^{\descChoice}_{m,k}Q &=-\cubecell{v/\descChoice(v)}{\sigma/\descChoice(v)} +\sum_{i=1}^k(-1)^{i+1} \cubecell{w_i}{\descChoice(v)/a_i,\sigma_i}.
    \end{align}
    Adding~\cref[commasep]{eq:det-push-a,eq:det-push-b}, and moving the first term to the left gives
    \begin{align}%
        \label{eq:det-push-a-b}
        \begin{split}
             & C(Q) +\cubecell{v/\descChoice(v)}{\sigma/\descChoice(v)} \\
             & \qquad\qquad= \sum_{i=1}^k(-1)^{i+1} \left( \cubecell{w_i}{\descChoice(v)/a_i,\sigma_i} -\linkcontraction_{w_i,\descChoice(w_i),k-1} \cubecell{w_i}{\sigma_i} \right).
        \end{split}
    \end{align}
    If $a_i\notin\defectset{Q}{\descChoice}$, then $\descChoice(w_i)=\descChoice(v)/a_i$ by~\cref{lem:hyperplane-transport:transport}.
    This edge is compatible with $\sigma_i$,
    so~\cref{lem:lifted-density-contraction-bounds:cone} shows that the corresponding summand on the right-hand side of~\cref{eq:det-push-a-b} is zero.
    If $a_i\in\defectset{Q}{\descChoice}$, the same summand has norm at most $1+M$ by~\cref{lem:lifted-density-contraction-bounds:norm}.
    Define
    \[
        \principalpart{t}{\descChoice} \defq \;\sum_{\mathclap{\substack{Q\in\supp(g)\\Q=\cubecell{v}{\sigma}}}}\; t_Q\cubecell{v/\descChoice(v)}{\sigma/\descChoice(v)},
    \]
    and $\residuepart{t}{\descChoice} \defq \morsepush^{\descChoice}_{m,k}t - \principalpart{t}{\descChoice}$,
    so $\morsepush^{\descChoice}_{m,k}t = \principalpart{t}{\descChoice} + \residuepart{t}{\descChoice}$.
    Because $\bottomboundary t=0$, we have $C(t)=\bottomboundary\linkcontraction^{\descChoice}_{m,k}t=-\morsepush^{\descChoice}_{m,k}t$,
    so $\residuepart{t}{\descChoice} = -C(t) - \principalpart{t}{\descChoice} = -(C(b) + C(g) + \principalpart{t}{\descChoice})$.

    Summing~\cref{eq:det-push-a-b} over $t_Q Q$ for $Q \in \supp(g)$ gives
    \begin{equation}%
        \label{eq:det-push-norm-a}
        \norm{C(g)+\principalpart{t}{\descChoice}} \leq (1+M)\defectweight{g}{\descChoice}.
    \end{equation}
    By the norm estimates on $\bottomboundary$ and $\linkcontraction^{\descChoice}_{m-1,k-1}$
    from~\cref{lem:cubical-boundary-splitting:norm,lem:lifted-density-contraction-bounds:norm}, we have
    $\norm{\bottomboundary\linkcontraction^{\descChoice}_{m,k}b} \leq (k+1)M\norm{b}$
    and $\norm{\linkcontraction^{\descChoice}_{m-1,k-1}\bottomboundary b} \leq kM\norm{b}$.
    Therefore,
    \begin{equation}%
        \label{eq:det-push-norm-b}
        \begin{aligned}
            \norm{C(b)}
             & = \norm{\bottomboundary\linkcontraction^{\descChoice}_{m,k}b +\linkcontraction^{\descChoice}_{m-1,k-1}\bottomboundary b} \\
             & \leq (k+1)M\norm{b} + kM\norm{b} =(2k+1)M\norm{b}.
        \end{aligned}
    \end{equation}
    Consequently, \cref{eq:det-push-norm-a,eq:det-push-norm-b} yield
    \begin{align}%
        \label{eq:det-push-norm-ab}
        \norm{\residuepart{t}{\descChoice}}
        = \norm{C(g) + \principalpart{t}{\descChoice} + C(b)}
        \leq (1+M)\defectweight{g}{\descChoice}+(2k+1)M\norm{b}.
    \end{align}
    Using~\cref{rem:good-bad-part-estimate} and the easy inequality $\norm{\principalpart{t}{\descChoice}} \leq \norm{g}$, this implies that
    \begin{align*}
        \norm{\morsepush^{\descChoice}_{m,k}t}
         &\leq \norm{g} +(1+M)\defectweight{g}{\descChoice} +(2k+1)M\norm{b} \\
         &\leq \norm{g} + \norm{b} + (2k+1)M(\defectweight{g}{\descChoice} + \defectweight{b}{\descChoice}) \\
         &= \norm{t} + (2k+1)M(\defectweight{t}{\descChoice}).\qedhere
    \end{align*}
\end{proof}

In view of~\cref{lem:deterministic-pushing-estimate}
and the discussion at the beginning of this~\lcnamecref{sec:probabilistic-pushing},
the strategy is to find a uniform constant $A'_k > 0$ so that
for each $\bottomboundary$\=/cycle $t \in \heightmodule{m}{k}$ there is a descending choice satisfying
\begin{align}%
    \label{eq:goal-estimate-defect}
    \defectweight{t}{\descChoice} \leq \frac{A'_k}{m} \norm{t}.
\end{align}
In order to do so, we study a suitable collection of such choices at once and show that,
under a random selection from this collection, there must be one choice satisfying~\cref{eq:goal-estimate-defect}.

\begin{defn}%
    \label{def:random-order}
    Let $\calH$ be a finite set.
    For a total order $\precOrd$ on $\calH$ and a non-empty subset $A\subseteq\calH$, write $\leastelement{\precOrd}{A}$ for the $\precOrd$-least element of $A$.

    Let $\totalorders{\calH}$ be the set of all total orders on $\calH$.
    We equip this set with the uniform probability measure.
    A \emph{random order on $\calH$} is a random element of this finite probability space.
\end{defn}

\begin{lem}%
    \label{lem:random-order-minima}
    Let $A$ and $B$ be non-empty subsets of a finite set $\calH$. The uniform probability measure on $\totalorders{\calH}$ satisfies
    \begin{refenum}
        \item\label{lem:random-order-minima:uniform}
              $\PP(\leastelement{\precOrd}{A}=\eta)=\size{A}^{-1}$ for every $\eta\in A$;
        \item\label{lem:random-order-minima:disagreement}
              \(\PP\bigl(\leastelement{\precOrd}{A} \neq\leastelement{\precOrd}{B}\bigr) = \size{A\symdiff B} \cdot \size{A\cup B}^{-1}\).
    \end{refenum}
\end{lem}
\begin{proof}
    The restriction of a uniformly chosen total order to $A$ is uniformly distributed among the total orders on $A$.
    Every element of $A$ is therefore equally likely to be the least element, which proves~\localref{lem:random-order-minima:uniform}.
    The least elements of $A$ and $B$ agree precisely when the least element of $A\cup B$ belongs to $A\cap B$.
    Therefore, applying~\localref{lem:random-order-minima:uniform} to $A\cup B$ proves~\localref{lem:random-order-minima:disagreement}.
\end{proof}

\begin{defn}[Descending choice induced by an order]%
    \label{def:ordered-descending-choice}
    Let $W\subseteq \superlevel{X}{a_0}^{(0)}$ be finite and set
    $\hyperplanesupport{W} \defq \bigcup_{u\in W}\densityhyperplanes{u}$.
    Let $\precOrd\in\totalorders{\hyperplanesupport{W}}$ be an order.
    It induces a descending choice $\descChoice_{\precOrd}$ on $W$ defined by choosing $\descChoice_{\precOrd}(u)$ to be
    the unique element in $\densityset{u} \subseteq \lkdown(u)$ such that
    \(\associatedhyperplane{\descChoice_{\precOrd}}{u} = \leastelement{\precOrd}{\densityhyperplanes{u}}\).
\end{defn}

\begin{lem}%
    \label{lem:random-order-defect-estimate}
    Suppose that the large overlap condition is satisfied with $\overlapconstant > 0$.
    Let $k\geq1$, let $m>a_0+1$, and let $t\in\heightmodule{m}{k}$.
    Let $\precOrd$ be a uniformly random element in $\totalorders{\hyperplanesupport{\choicevertices{t}}}$.
    The induced descending choice $\descChoice_{\precOrd}$ satisfies
    \[
        \expect\bigl[ \defectweight{t}{\descChoice_{\precOrd}} \bigr] \leq \frac{k\overlapconstant}{m}\norm{t}\mpct{.}
    \]
\end{lem}
\begin{proof}
    Fix a basis cube $Q=\cubecell{v}{\sigma}\in\supp(t)$ and an element $a\in \vertices{\sigma}$, and put $w=v/a$.
    By~\cref{def:large-overlap-descending-links,lem:random-order-minima:disagreement},
    \[
        \PP\bigl(a\in\defectset{Q}{\descChoice_{\precOrd}}\bigr)
        = \PP\bigl(\associatedhyperplane{\descChoice_{\precOrd}}{v} \neq \associatedhyperplane{\descChoice_{\precOrd}}{w}\bigr)
        = \frac{\size{\densityhyperplanes{v}\symdiff\densityhyperplanes{w}}}
        {\size{\densityhyperplanes{v}\cup\densityhyperplanes{w}}}
        \leq \frac{\overlapconstant}{m}.
    \]
    We have $\size{\defectset{Q}{\descChoice_{\precOrd}}} = \sum_{a \in \vertices{\sigma}} \mathbb{1}_{\defectset{Q}{\descChoice_{\precOrd}}}(a)$,
    where $\mathbb{1}_{\defectset{Q}{\descChoice_{\precOrd}}}$ denotes the indicator function of the set $\defectset{Q}{\descChoice_{\precOrd}}$.
    Applying linearity of $\expect$ to the definition of $\defectweight{t}{\descChoice_{\precOrd}}$
    together with $\size{\vertices{\sigma}}=k$ therefore gives
    \begin{align*}
        \expect\bigl[\defectweight{t}{\descChoice_{\precOrd}}\bigr]
        = \expect\Bigl[
                     \sum_Q \abs{t_Q} \,\sum_{\mathclap{a\in \vertices{\sigma}}}\; \mathbb{1}_{\defectset{Q}{\descChoice_{\precOrd}}}(a)
                     \Bigr]
         &= \sum_{Q} \abs{t_Q} \,\sum_{\mathclap{a\in \vertices{\sigma}}}\; \PP\bigl(a\in\defectset{Q}{\descChoice_{\precOrd}}\bigr) \\
         &\leq \frac{k\overlapconstant}{m} \sum_{Q}\abs{t_Q}
        = \frac{k\overlapconstant}{m}\norm{t}.\qedhere
    \end{align*}
\end{proof}

\begin{lem}%
    \label{lem:probabilistic-pushing}
    Fix $k\geq1$ and suppose that $X$ has strong uniform density of descending links
    and satisfies the large overlap condition.
    There exist an integer $a_k$ and constants $A_k,B_k>0$
    such that for every $m > a_k$ the following holds.
    \begin{refenum}
        \item\label{lem:probabilistic-pushing:pushing}
              Let $t\in\heightmodule{m}{k}$ be a chain satisfying $\bottomboundary t=0$.
              There is a descending choice $\descChoice$ on $\choicevertices{t}$ such that
              $\norm{\morsepush^{\descChoice}_{m,k}t}\leq(1+A_k/m)\norm{t}$.
        \item\label{lem:probabilistic-pushing:uniform}
              Let $x\in\heightmodule{m}{k}$.
              Then $\norm{\morsepush^{\descChoice}_{m,k}x}\leq B_k\norm{x}$ for every descending choice $\descChoice$ on $\choicevertices{x}$.
    \end{refenum}
\end{lem}
\begin{proof}
    Let $a_k$ and $M_{k-1} \geq 1$ be the constants obtained
    from~\cref{lem:lifted-density-contraction-bounds} with $p=k-1$.
    Since $m>a_k$, the conclusions of~\cref{lem:lifted-density-contraction-bounds} hold at heights $m$ and $m-1$.
    Put $L_k=(2k+1)M_{k-1}$, $A_k=kL_k\overlapconstant$, and $B_k=(k+1)M_{k-1}$.

    Ad~\localref{lem:probabilistic-pushing:pushing}.
    Choose $\precOrd$ uniformly from $\totalorders{\hyperplanesupport{\choicevertices{t}}}$ and let $\descChoice_{\precOrd}$ be its induced descending choice.
    By~\cref[commasep]{lem:deterministic-pushing-estimate,lem:random-order-defect-estimate}, and linearity of expectation,
    \begin{align*}
        \expect\Bigl[\norm{\morsepush^{\descChoice_{\precOrd}}_{m,k}t}\Bigr]
         &\leq \norm{t} +L_k\expect\bigl[\defectweight{t}{\descChoice_{\precOrd}}\bigr] \\
         &\leq \left( 1+kL_k \frac{\overlapconstant}{m} \right)\norm{t}
        = \left(1+\frac{A_k}{m}\right)\norm{t}.
    \end{align*}
    Since this expectation is an average over the finite set $\totalorders{\hyperplanesupport{\choicevertices{t}}}$,
    there is an order $\precOrd$ such that $\descChoice\defq\descChoice_{\precOrd}$
    satisfies~\localref{lem:probabilistic-pushing:pushing}.

    Ad~\localref{lem:probabilistic-pushing:uniform}.
    The norm estimates in~\cref{lem:cubical-boundary-splitting:norm,lem:lifted-density-contraction-bounds:norm} imply
    the required estimate.
\end{proof}

\begin{rem}%
    \label{rem:extend-descending-choice}
    By~\cref{def:one-level-pushing-operator,def:descending-choice},
    a descending choice from~\cref{lem:probabilistic-pushing}
    can be extended to a descending choice on all vertices of $\superlevel{X}{a_0}$.
\end{rem}

\subsection{Pushing fillings into a sublevel complex}

Applying the norm bounds from~\cref{lem:probabilistic-pushing} to the pushing procedure described
in~\cref{rem:pushing-construction} yields bounds on the norm
of pushed-down chains whose norms grow polynomially in the initial height.

\begin{prop}%
    \label{prop:quantitative-pushing}
    Fix $k\geq1$ and suppose that $X$ has strong uniform density of descending links and satisfies the large overlap condition.
    There exist an integer $a_k$ and constants $P_k,\delta_k>0$ with the following property.
    Let $a,H\in\bbZ$ satisfy $H\geq a\geq a_k$, and let $b\in C_k(X)$ be a chain such that $b\in C_k(\sublevel{X}{H})$
    and $\supp(\partial b)\subseteq\sublevel{X}{a}$.
    Then there are chains $u\in C_{k+1}(X)$ and $b_a\in C_k(\sublevel{X}{a})$ satisfying
    \begin{refenum}
        \item\label{prop:quantitative-pushing:boundary}
              $b_a=b-\partial u$.
        \item\label{prop:quantitative-pushing:norm}
              $\norm{b_a}\leq P_k \cdot \bigl((1+H)/(1+a)\bigr)^{\delta_k} \cdot \norm{b}$.
    \end{refenum}
\end{prop}
\begin{proof}
    Fix a height $a_0$ above which the descending links satisfy strong uniform density and the large overlap condition.
    Let $a_k$, $A_k$, and $B_k$ be supplied by~\cref{lem:probabilistic-pushing}.
    We set $\delta_k\defq A_k$ and $P_k\defq B_k$.
    After possibly increasing these constants, we may assume that both are at least $1$.

    If $H=a$, take $u=0$ and $b_a=b$.
    We may therefore assume that $H>a$.
    The height grading gives a decomposition
    \[
        b=x_{\leq a}+\sum_{m=a+1}^{H}x_m, \qquad x_{\leq a}\in C_k(\sublevel{X}{a}), \quad x_m\in\heightmodule{m}{k}.
    \]
    We construct chains $b^{(m)} \in C_k(X)$ and $t_m \in \heightmodule{m}{k}$ for $a\leq m\leq H$, together with chains $u_m \in \heightmodule{m}{k+1}$ for $a<m\leq H$, by descending induction so that
    \begin{condrefenum}
        \item\label{prop:quantitative-pushing:stage}
              $b^{(m)}=x_{\leq a}+\sum_{r=a+1}^{m}x_r+t_m$.
        \item\label{prop:quantitative-pushing:relative}
              $b^{(m)}=b-\partial\sum_{r=m+1}^{H}u_r$.
        \item\label{prop:quantitative-pushing:pushed-cycle}
              $\bottomboundary t_m=0$.
        \item\label{prop:quantitative-pushing:pushing-norm}
              \(\norm{t_m} \leq B_k\sum_{r=m+1}^{H}\norm{x_r} \prod_{l=m+1}^{r-1}(1+A_k/l)\).
    \end{condrefenum}
    For $m=H$, we can take $b^{(H)}=b$ and $t_H=0$.
    The sums in~\localref{prop:quantitative-pushing:relative,prop:quantitative-pushing:pushing-norm} are then empty, so these choices satisfy all four conditions.

    Let $a<m\leq H$ and suppose that $b^{(m)}$ and $t_m$ have been constructed.
    Put $y_m=x_m+t_m$, which is the component of $b^{(m)}$ at height $m$ by~\localref{prop:quantitative-pushing:stage}.
    By~\localref{prop:quantitative-pushing:relative}, the boundary of $b^{(m)}$ equals $\partial b$ and is therefore supported in $\sublevel{X}{a}$.
    Since $b^{(m)}$ has no component above height $m$, the component of its boundary at height $m$ is $\topboundary y_m$.
    It follows that $\topboundary y_m=0$.

    Since $m\geq a+1>a_k$,~\localref{prop:quantitative-pushing:pushed-cycle} allows us to apply~\cref{lem:probabilistic-pushing:pushing} to $t_m$.
    Let $\descChoice_m$ be the descending choice supplied there.
    By~\cref{rem:extend-descending-choice}, we can extend
    $\descChoice_m$ to all vertices of $X$ so that $\morsepush^{\descChoice_m}_{m,k}y_m$ is defined.
    Set
    $u_m\defq\linkcontraction^{\descChoice_m}_{m,k}y_m$.
    The contraction identity in~\cref{rem:levelwise-contraction} gives
    \begin{equation}%
        \label{eq:top-boundary-u-m}
        \topboundary u_m =y_m-\linkcontraction^{\descChoice_m}_{m,k-1}\topboundary y_m =y_m.
    \end{equation}
    Define $b^{(m-1)}\defq b^{(m)}-\partial u_m$ and
    $t_{m-1}\defq\morsepush^{\descChoice_m}_{m,k}y_m=-\bottomboundary u_m$.
    Using $\partial=\topboundary+\bottomboundary$,
    the induction hypothesis~\localref{prop:quantitative-pushing:stage} for $b^{(m)}$, and~\cref{eq:top-boundary-u-m}, we obtain
    \begin{align*}
        b^{(m-1)}
        = b^{(m)}-\partial u_m
         &= x_{\leq a}+\sum_{r=a+1}^{m}x_r+t_m - y_m - \bottomboundary u_m \\
         &=x_{\leq a}+\sum_{r=a+1}^{m-1}x_r+t_{m-1}.
    \end{align*}
    Thus,~\localref{prop:quantitative-pushing:stage} holds for height $m-1$.
    The definition of $b^{(m-1)}$ and~\localref{prop:quantitative-pushing:relative} at height $m$ give~\localref{prop:quantitative-pushing:relative} at height $m-1$.
    Moreover,~\cref{lem:cubical-boundary-splitting:squares} gives
    $\bottomboundary t_{m-1} =-(\bottomboundary)^2u_m =0$,
    which proves~\localref{prop:quantitative-pushing:pushed-cycle} at height $m-1$.

    The estimates in~\cref[commasep]{lem:probabilistic-pushing:pushing,lem:probabilistic-pushing:uniform}, and~\localref{prop:quantitative-pushing:pushing-norm} at height $m$ give
    \begin{align*}
        \norm{t_{m-1}} = \norm{\morsepush^{\descChoice_m}_{m,k}y_m}
         &\leq \norm{\morsepush^{\descChoice_m}_{m,k}t_m} +\norm{\morsepush^{\descChoice_m}_{m,k}x_m} \\
         &\leq (1+A_k/m)\norm{t_m}+B_k\norm{x_m} \\
         &\leq B_k\sum_{r=m+1}^{H}\norm{x_r} \prod_{l=m}^{r-1}(1+A_k/l) +B_k\norm{x_m} \\
         &= B_k\sum_{r=m}^{H}\norm{x_r} \prod_{l=m}^{r-1}(1+A_k/l).
    \end{align*}
    Hence,~\localref{prop:quantitative-pushing:pushing-norm} holds at height $m-1$.

    Ultimately,~\localref{prop:quantitative-pushing:pushing-norm} at height $a$ yields
    \[
        \norm{t_a} \leq B_k\sum_{r=a+1}^{H}\norm{x_r} \prod_{m=a+1}^{r-1}\Bigl(1+ \frac{A_k}{m}\Bigr).
    \]
    Since $\delta_k\geq1$ and $\delta_k\geq A_k$, Bernoulli's inequality gives
    \[
        1+A_km^{-1} \leq 1+\delta_km^{-1} \leq (1+m^{-1})^{\delta_k} = \Bigl(\frac{m+1}{m}\Bigr)^{\delta_k}.
    \]
    Consequently,
    \begin{align*}
        \prod_{m=a+1}^{r-1}(1+A_km^{-1}) \leq \prod_{m=a+1}^{r-1} \Bigl(\frac{m+1}{m}\Bigr)^{\delta_k} = \Bigl(\frac{r}{a+1}\Bigr)^{\delta_k} \leq \Bigl(\frac{1+H}{1+a}\Bigr)^{\delta_k}.
    \end{align*}
    It follows that
    \[
        \norm{t_a} \leq B_k \cdot \Bigl(\frac{1+H}{1+a}\Bigr)^{\delta_k} \sum_{r=a+1}^{H}\norm{x_r}.
    \]

    Put $u\defq\sum_{m=a+1}^{H}u_m$ and $b_a\defq b^{(a)}=x_{\leq a}+t_a$.
    In particular, $b_a$ belongs to $C_k(\sublevel{X}{a})$.
    Condition~\localref{prop:quantitative-pushing:relative} at height $a$ gives $b_a=b-\partial u$, which proves~\localref{prop:quantitative-pushing:boundary}.
    Since $(1+H)/(1+a)\geq1$, the preceding estimate and the definition of $P_k$ give
    \begin{align*}
        \norm{b_a} &\leq \norm{x_{\leq a}}+\norm{t_a} \\
                   &\leq P_k \cdot \left(\frac{1+H}{1+a}\right)^{\delta_k} \cdot \Bigl( \norm{x_{\leq a}} +\sum_{r=a+1}^{H}\norm{x_r} \Bigr) \\
                   &\leq P_k \cdot \left(\frac{1+H}{1+a}\right)^{\delta_k} \cdot \norm{b}.
    \end{align*}
    This proves~\localref{prop:quantitative-pushing:norm}.
\end{proof}

The following classical result about \CAT{0} cube complexes lets us deduce polynomiality
of filling functions for sufficiently high sublevel sets $\sublevel{X}{a}$.

\begin{thm}[{\cites{Wenger2005,ABDDY13}}]%
    \label{thm:CAT0-filling}%
    Let $X$ be a \CAT{0} cube complex and let $k\geq2$.
    Then its filling
    function satisfies $\filling{X}{k}(v) \preccurlyeq v^{\frac{k}{k-1}}$.
\end{thm}

The precise cellular statement used here, with integral coefficients and without a bound on the dimension of $X$, does not seem to have an explicit proof in the literature.
The statement does, however, appear in~\cite{li+manin2021}*{Proposition~2.3}, so a proof
may be known to experts in the field.
The finite-dimensional case follows immediately from the work of~\cite{Wenger2005} and~\cite{ABDDY13}.
For completeness, we include a sketch of
the proof of the above result in the infinite-dimensional case.

\begin{proof}
    Let $z$ be a cellular $(k-1)$\=/cycle in $X$.
    The cubical convex hull $K$
    of the vertices in $\supp(z)$ is finite, since only finitely many
    hyperplanes separate these vertices.
    By~\cite{Wright2012}*{Theorem~3.6}, $K$ embeds into a cube $Q_N \defq [0,1]^N$
    and there is a retraction
    $r_\ast\colon C_\ast(Q_N,\bbZ)\to C_\ast(K,\bbZ)$ up to degree $k$
    such that $\norm{r_\ast} \leq 1$.

    In~\cite{Dotterrer2015}*{Theorem~3.2}, Dotterrer proves
    that there exists a constant $D_k$, independent of $N$,
    such that every $(k-1)$\=/cycle $c$ in $C_{k-1}(Q_N,\bbZ/2\bbZ)$ admits
    a filling $b\in C_k(Q_N,\bbZ/2\bbZ)$ with norm at most
    $D_k\norm{c}^{k/(k-1)}$.
    Dotterrer's argument applies almost verbatim with $\bbZ$ coefficients rather than
    $\bbZ/2\bbZ$ coefficients and hence produces an integral filling $b$ of $z$ with
    $\norm{b}\leq D_k\norm{z}^{k/(k-1)}$.
    Applying the chain retraction $r_\ast$ to $b$ gives the required filling in $K\subseteq X$.
\end{proof}

\begin{thm}[store=main-sublevel-fillings]%
    \label{cor:CAT0-ambient-fillings}
    Let $X$ be a \CAT{0} cube complex equipped with a discrete Morse function $h \colon X \to \bbR$,
    and suppose that $X$ has strong uniform density of descending links and satisfies the large overlap condition.
    For every $k\geq2$, there is an integer $a_k$ such that
    $\filling{\sublevel{X}{a}}{k}$ is polynomial for every $a\geq a_k$.
\end{thm}
\begin{proof}
    Fix $k\geq2$.
    Let $a_k$, $P_k$, and $\delta_k$ be supplied by~\cref{prop:quantitative-pushing}.
    By~\cref{thm:CAT0-filling}, there is a polynomial $V_k$ such that every cellular $(k-1)$\=/cycle $z$ in $X$ has a filling $b\in C_k(X)$ satisfying $\norm{b}\leq V_k(\norm{z})$.
    Since every polynomial is $\approx$\=/equivalent to a superadditive polynomial, it suffices to consider cycles with connected support.

    Fix $a\geq a_k$, let $z$ be a cellular $(k-1)$\=/cycle in $\sublevel{X}{a}$ with connected support, and put $n=\norm{z}$.
    Choose a filling $b$ of $z$ as above.
    Because the support of $z$ is connected, we may assume that the cubes in $\supp(b)$ form a connected subcomplex of $X$.
    Since $\size{\supp(b)}\leq\norm{b}$, every cube in $\supp(b)$ has maximal height at most
    $a+\norm{b}\leq a+V_k(n)$.

    Applying~\cref{prop:quantitative-pushing} to $b$ gives a filling $b_a$ in $\sublevel{X}{a}$ such that
    \[
        \norm{b_a} \leq P_k \cdot \left( \frac{1+a+V_k(n)}{1+a} \right)^{\delta_k} \cdot V_k(n).
    \]
    The right-hand side is a polynomial in $n$, which proves the assertion.
\end{proof}

\begin{rem}%
    \label{rem:improved-height-estimate}
    The height estimate of the filling $b$ in the above proof is not optimal.
    One could show, using cellular versions of the techniques in~\cite{Guth2006}*{Lemma~6},
    that the filling $b$ has height at most $a + O(n^{1/(k-1)})$.
    This would improve the norm bound to $\norm{b_a}\preccurlyeq n^{(k + \delta_k)/(k-1)}$.
\end{rem}

\section{Application to Higman--Thompson groups}%
\label{sec:Higman-Thompson-strong-uniform-density}

Fix integers $d\geq2$ and $r\geq1$.
Recall that $F_{d,r}$, $T_{d,r}$, and $V_{d,r}$ are the $d$\=/ary Higman--Thompson groups with $r$ roots.
For $(d,r)=(2,1)$, these are Thompson's groups $F$, $T$, and $V$, respectively.
We refer to~\cite{SkipperWu2023}*{Section~3.1} for the definition, as we do not require it here.

For $\sfA\in\{F,T,V\}$, let $\FScomplex{\sfA}{d}{r}$ be the Stein--Farley cube complex for $\sfA_{d,r}$.
The cubical construction and the proof that this cube complex is \CAT{0} are due to Farley~\cite{Farley2005}.

We use only the following aspects of the construction, which are
recorded in Sections~4A and~4E of~\cite{Brown1987} and in
Sections~2.1, 3.4, and~5.4 of~\cite{Witzel2019}.
There is an $\sfA_{d,r}$\=/invariant rank function $\vertexrank{\placeholder}$ on the vertices of
$\FScomplex{\sfA}{d}{r}$ whose values are congruent to $r$ modulo $d-1$.
The rank extends affinely over every cube and changes by $d-1$ along each edge.
Consequently, the function
\[
    h(v)\defq\frac{\vertexrank{v}-r}{d-1}
\]
on vertices extends affinely to a discrete Morse function in the sense of~\cref{def:discrete-Morse-function}.
A vertex of height $m$ has rank
$\heightrank{m}\defq r+(d-1)m$.

By Section~4E of~\cite{Brown1987}, vertices of each fixed rank form a single orbit,
so every sublevel set is cocompact.
The action of $F_{d,r}$ is free, while
the stabilizers of a vertex of rank $n$ under $T_{d,r}$ and $V_{d,r}$ are
$\bbZ/n\bbZ$ and $S_n$, respectively.
Thus, all cell stabilizers are finite.

To verify strong uniform density and the large overlap condition,
we need to understand the descending links of $\FScomplex{\sfA}{d}{r}$.
In all cases, they are certain (ordered) matching complexes.

\begin{defn}%
    \label{def:Higman-Thompson-matching-complexes}
    Let $V$ be a set and $\calE \subseteq \powerset{V}$ a collection
    of subsets of $V$.
    Its \emph{matching complex} $\matchingcomplex{\calE}$ has the elements of $\calE$ as vertices,
    and a collection of vertices spans a simplex precisely when their corresponding subsets of $V$ are pairwise disjoint.

    If $\calE \subseteq V^d$ is a collection of $d$\=/tuples, its \emph{ordered matching complex} $\orderedmatchingcomplex{\calE}$
    is defined in the same way. We are interested in the following three examples.
    \begin{refenum}
        \item Let $P_n$ denote the \emph{path graph} on $n$
              vertices, that is, the edges are given by $\{i,i+1\}$
              for $1 \leq i < n$.
              Define $\pathhypergraph{d}{n} \subseteq \powerset{\{1,\ldots,n\}}$
              to be the collection of intervals in $P_n$ containing precisely $d$ vertices.
        \item Let $C_n$ denote the \emph{cycle graph} on $n$
              vertices, that is, two vertices $i$ and $j$ are connected by an edge if $i \pm 1 \equiv j \mod n$.
              Define $\cyclehypergraph{d}{n} \subseteq \powerset{\{1,\ldots,n\}}$
              to be the collection of intervals in $C_n$ containing precisely $d$ vertices.
        \item Let $B$ be a set.
              Define $\completehypergraph*{d}{B} \subseteq \powerset{B}$ to be the collection of all $d$-element subsets of $B$, and
              $\completeorderedhypergraph*{d}{B} \subseteq B^d$ to be the collection of all (ordered) $d$\=/tuples with pairwise distinct entries.
              If $B = \{1,\ldots,n\}$, we also denote these by $\completehypergraph{d}{n}$
              and $\completeorderedhypergraph{d}{n}$.
    \end{refenum}
\end{defn}

\begin{rem}
    A subset $\calE \subseteq \powerset{V}$ is also called a \emph{hypergraph}.
    Moreover, for $d=2$, there are natural identifications
    $\pathgraph{n} \cong \pathhypergraph{2}{n}$ and $\cyclegraph{n} \cong \cyclehypergraph{2}{n}$.
    Their matching complexes are depicted in~\cref{fig:matching-complex-line-cycle}.

    \begin{figure}
        \centering
        \pgfmathsetlengthmacro{\vertexLabelDistance}{2pt}%
\colorlet{complex face color}{black!10}%
\colorlet{chosen face color}{Yellow}%
\colorlet{choice one color}{Red}%
\colorlet{choice two color}{Blue}%
\colorlet{choice three color}{Green}%
\tikzset{%
    graph vertex/.style={circle,fill=black,inner sep=1pt},%
    graph edge/.style={thick},%
    choice one edge/.style={draw=choice one color,ultra thick},%
    choice two edge/.style={draw=choice two color,ultra thick},%
    choice three edge/.style={draw=choice three color,ultra thick},%
    complex edge/.style={graph edge},%
    complex face/.style={fill=complex face color},%
    chosen face/.style={fill=chosen face color},%
    complex vertex/.style={graph vertex},%
    choice one vertex/.style={fill=choice one color,draw=choice one color,inner sep=1.5pt},%
    choice two vertex/.style={fill=choice two color,draw=choice two color,inner sep=1.5pt},%
    choice three vertex/.style={fill=choice three color,draw=choice three color,inner sep=1.5pt},%
    complex label/.style={font=\small},%
}%
\begin{tikzpicture}[baseline=(complex center)]
    \pgfmathsetlengthmacro{\pathEdgeLength}{7.5mm}
    \pgfmathsetlengthmacro{\diagramSeparation}{34mm}
    \pgfmathsetlengthmacro{\complexRadius}{15mm}

    \coordinate (path center) at (0,0);
    \begin{scope}
        \foreach \i in {1,...,6} {
                \coordinate (p\i) at ($(path center)+({(\i-3.5)*\pathEdgeLength},0)$);
            }
        \node[left=1mm of p1] {$\pathgraph{6}$};
        \begin{scope}[graph edge]
            \foreach \i/\j in {1/2,2/3,3/4,4/5,5/6} {
                    \draw (p\i)--(p\j);
                }
        \end{scope}
        \draw[choice one edge] (p1)--(p2);
        \draw[choice two edge] (p3)--(p4);
        \draw[choice three edge] (p5)--(p6);
        \begin{scope}[every node/.style={graph vertex}]
            \foreach \i in {1,...,6} {
                    \node at (p\i) {};
                }
        \end{scope}
        \foreach \i in {1,...,6} {
                \node[below=\vertexLabelDistance] at (p\i) {$\scriptstyle\i$};
            }
    \end{scope}

    \coordinate (complex center) at ($(path center)+(0,-\diagramSeparation)$);
    \foreach \i/\a in {3/90,5/18,2/-54,4/-126,1/162} {
            \coordinate (m\i) at ($(complex center)+(\a:\complexRadius)$);
        }
    \node[below] at (complex center) {$\matchingcomplex{\pathgraph{6}}$};
    \begin{scope}
        \fill[chosen face] (m1)--(m3)--(m5)--cycle;
        \begin{scope}[complex edge]
            \foreach \i/\j in {1/3,1/5,3/5,1/4,4/2,2/5} {
                    \draw (m\i)--(m\j);
                }
        \end{scope}

        \begin{scope}[every node/.style={complex vertex}]
            \node[choice one vertex] at (m1) {};
            \node at (m2) {};
            \node[choice two vertex] at (m3) {};
            \node at (m4) {};
            \node[choice three vertex] at (m5) {};
        \end{scope}

        \begin{scope}[every node/.style={complex label}]
            \node[left=\vertexLabelDistance] at (m1) {$\{1,2\}$};
            \node[below=\vertexLabelDistance] at (m2) {$\{2,3\}$};
            \node[above=\vertexLabelDistance] at (m3) {$\{3,4\}$};
            \node[below=\vertexLabelDistance] at (m4) {$\{4,5\}$};
            \node[right=\vertexLabelDistance] at (m5) {$\{5,6\}$};
        \end{scope}
    \end{scope}
\end{tikzpicture}\quad%
\begin{tikzpicture}[baseline=(complex center)]
    \pgfmathsetlengthmacro{\cycleRadius}{10.5mm}
    \pgfmathsetlengthmacro{\cycleLabelRadius}{13.5mm}
    \pgfmathsetlengthmacro{\diagramSeparation}{34mm}
    \pgfmathsetlengthmacro{\complexHalfWidth}{20mm}
    \pgfmathsetlengthmacro{\complexHalfHeight}{10.5mm}
    \pgfmathsetlengthmacro{\complexInnerOffset}{11.5mm}

    \coordinate (cycle center) at (0,0);
    \coordinate (cycle name anchor) at ($(cycle center)+(-\cycleRadius,0)$);
    \begin{scope}
        \foreach \i/\a in {1/90,2/30,3/-30,4/-90,5/-150,6/150} {
                \coordinate (c\i) at ($(cycle center)+(\a:\cycleRadius)$);
                \coordinate (cl\i) at ($(cycle center)+(\a:\cycleLabelRadius)$);
            }
        \node[left=1mm] at (cycle name anchor) {$\cyclegraph{6}$};
        \begin{scope}[graph edge]
            \foreach \i/\j in {1/2,2/3,3/4,4/5,5/6,6/1} {
                    \draw (c\i)--(c\j);
                }
        \end{scope}
        \draw[choice one edge] (c1)--(c2);
        \draw[choice two edge] (c3)--(c4);
        \draw[choice three edge] (c5)--(c6);
        \begin{scope}[every node/.style={graph vertex}]
            \foreach \i in {1,...,6} {
                    \node at (c\i) {};
                }
        \end{scope}
        \foreach \i in {1,...,6} {
                \node at (cl\i) {$\scriptstyle\i$};
            }
    \end{scope}

    \coordinate (complex center) at ($(cycle center)+(0,-\diagramSeparation)$);
    \coordinate (n1) at ($(complex center)+(-\complexHalfWidth,0)$);
    \coordinate (n4) at ($(complex center)+(\complexHalfWidth,0)$);
    \coordinate (n3) at ($(complex center)+(-\complexInnerOffset,\complexHalfHeight)$);
    \coordinate (n5) at ($(complex center)+(-\complexInnerOffset,-\complexHalfHeight)$);
    \coordinate (n6) at ($(complex center)+(\complexInnerOffset,\complexHalfHeight)$);
    \coordinate (n2) at ($(complex center)+(\complexInnerOffset,-\complexHalfHeight)$);
    \begin{scope}
        \fill[complex face] (n2)--(n4)--(n6)--cycle;
        \fill[chosen face] (n1)--(n3)--(n5)--cycle;
        \begin{scope}[complex edge]
            \foreach \i/\j in {1/3,1/5,3/5,2/4,2/6,4/6,3/6,5/2} {
                    \draw (n\i)--(n\j);
                }
            \draw (n1) .. controls
            ($(n1)+(0,\complexHalfWidth)$) and
            ($(n4)+(0,\complexHalfWidth)$) .. (n4);
        \end{scope}

        \begin{scope}[every node/.style={complex vertex}]
            \node[choice one vertex] at (n1) {};
            \node at (n2) {};
            \node[choice two vertex] at (n3) {};
            \node at (n4) {};
            \node[choice three vertex] at (n5) {};
            \node at (n6) {};
        \end{scope}

        \begin{scope}[every node/.style={complex label}]
            \node[left=\vertexLabelDistance] (n1 label) at (n1) {$\{1,2\}$};
            \node[below=\vertexLabelDistance] at (n2) {$\{2,3\}$};
            \node[below right] at (n3) {$\{3,4\}$};
            \node[right=\vertexLabelDistance] at (n4) {$\{4,5\}$};
            \node[below=\vertexLabelDistance] at (n5) {$\{5,6\}$};
            \node[below left] at (n6) {$\{1,6\}$};
        \end{scope}
    \end{scope}
    \node[below] at (complex center) {$\matchingcomplex{\cyclegraph{6}}$};
\end{tikzpicture}

        \caption{%
            \label{fig:matching-complex-line-cycle}
            The complexes $\matchingcomplex{\pathgraph{6}}$ and $\matchingcomplex{\cyclegraph{6}}$.
            The coloured vertices span the highlighted simplices and correspond to the coloured matchings in the graphs.
        }
    \end{figure}
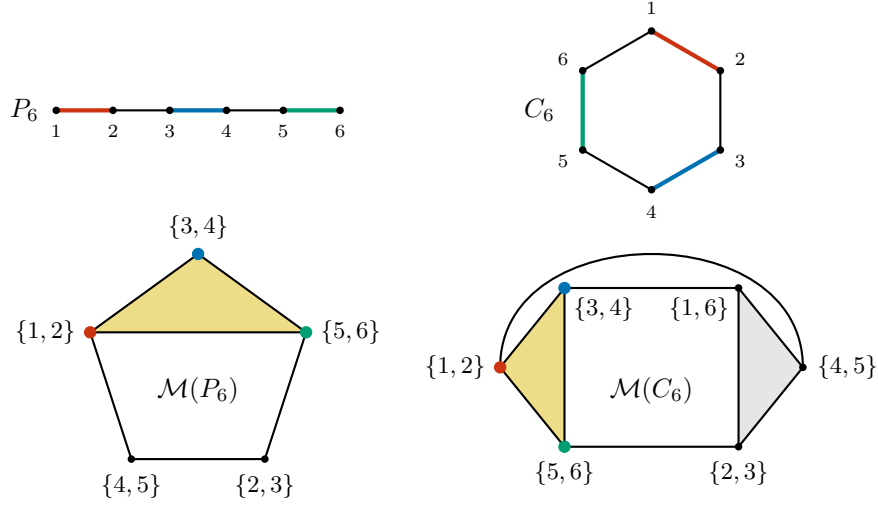
\end{rem}

\begin{prop}%
    \label{prop:Higman-Thompson-descending-links}
    Let $v$ be a vertex of rank $n \geq d$.
    Put $B_n\defq\{1,\ldots,n\}$.
    The descending link of $v$ is described as follows.
    \begin{refenum}
        \item\label{prop:Higman-Thompson-descending-links:F}
              For $\FScomplex{F}{d}{r}$, one has $\lkdown(v)\cong\matchingcomplex{\pathhypergraph{d}{n}}$.
        \item\label{prop:Higman-Thompson-descending-links:T}
              For $\FScomplex{T}{d}{r}$, one has $\lkdown(v)\cong\matchingcomplex{\cyclehypergraph{d}{n}}$.
        \item\label{prop:Higman-Thompson-descending-links:V}
              For $\FScomplex{V}{d}{r}$, one has
              $\lkdown(v)\cong\orderedmatchingcomplex{\completeorderedhypergraph{d}{n}}$.
    \end{refenum}
\end{prop}
\begin{proof}
    The categorical realization in~\cite{Witzel2019}*{Sections~2.1 and~5.4} presents the Higman--Thompson groups as fundamental groups of certain categories, where $r$ corresponds to the chosen base object in the corresponding category.
    The proof of~\cite{Witzel2019}*{Theorem~3.12}, together with~\cite{Witzel2019}*{Remark~3.17}, shows that descending links at vertices of a fixed rank can be identified independently of the choice of $r$.
    It therefore suffices to consider the case $r=1$.

    For $d=2$, the descriptions in~\localref{prop:Higman-Thompson-descending-links:F,prop:Higman-Thompson-descending-links:T} are precisely those in~\cite{Witzel2019}*{Lemmas~3.16 and~6.1}.
    It is not difficult to see that the same reasoning also works for $d > 2$.
    For~\localref{prop:Higman-Thompson-descending-links:V}, one checks that the description of $\lkdown(v)$ in~\cite{SkipperZaremsky2024}*{Example~3.2, Definition~4.1}
    admits a natural identification with $\orderedmatchingcomplex{\completeorderedhypergraph{d}{n}}$.
\end{proof}

It is well known that the connectivity of these descending links
tends to infinity as $n$ goes to infinity.
This can, for instance, be proved by using the uniform density condition~\cref{thm:Belk-Forrest-density} from Belk--Forrest (see~\cite{Witzel2019}*{Lemma~3.19}).
However, to apply our criterion for polynomial Dehn functions,
we also need to prove strong uniform density and the large overlap condition for these complexes.
In all three cases, we verify the conditions of~\cref{prop:full-descending-link-density-criterion}.

\begin{prop}%
    \label{prop:F-strong-uniform-density}
    The complex $\FScomplex{F}{d}{r}$ has strong uniform density of descending links and satisfies the large overlap condition.
\end{prop}
\begin{proof}
    Define $\altoverlapconst \defq 6$ and let $m\geq a_0 \defq 2$.
    A vertex of height $m$ has rank $\heightrank{m}=r+(d-1)m$.
    By~\cref{prop:Higman-Thompson-descending-links:F}, the vertices of its descending link are the intervals of length $d$ in a linearly ordered set of cardinality $\heightrank{m}$, and two distinct vertices are compatible precisely when their intervals are disjoint.
    A fixed interval meets at most $d-1$ intervals beginning immediately before it and at most $d-1$ intervals ending immediately after it.
    Hence, the descending link is uniformly $b_m$\=/dense and has $N_m$ vertices, where
    \[
        N_m\defq\heightrank{m}-d+1=(d-1)(m-1)+r, \qquad b_m\defq2(d-1).
    \]
    If $m>a_0$, then $N_m\geq(d-1)m/2$ and $\abs{N_m-N_{m-1}}=d-1$.
    Therefore,
    \[
        \frac{b_m+\abs{N_m-N_{m-1}}}{N_m}
        =\frac{3(d-1)}{N_m}
        \leq\frac{\altoverlapconst}{m}.
    \]
    This verifies the conditions of~\cref{prop:full-descending-link-density-criterion}.
\end{proof}

\begin{prop}%
    \label{prop:T-strong-uniform-density}
    The complex $\FScomplex{T}{d}{r}$ has strong uniform density of descending links and satisfies the large overlap condition.
\end{prop}
\begin{proof}
    The proof is the same as that of~\cref{prop:F-strong-uniform-density}, except that now
    $N_m = r+(d-1)m\geq(d-1)m$, since we consider the larger set of intervals in a cycle graph rather than a path graph.
    Since $\abs{N_m-N_{m-1}}=d-1$ also holds in this case and $N_m$ appears only in the denominator of the relevant ratio, the larger value of $N_m$ only improves the inequality.
\end{proof}

\begin{prop}%
    \label{prop:V-strong-uniform-density}
    The complex $\FScomplex{V}{d}{r}$ has strong uniform density of descending links and satisfies the large overlap condition.
\end{prop}
\begin{proof}
    Choose an integer $a_0\geq1$ such that $\heightrank{m}\geq2d$ for every $m\geq a_0$.
    Define $\altoverlapconst\defq2d^2/(d-1)$.
    Let $m\geq a_0$ and put $n\defq\heightrank{m}$.
    By~\cref{prop:Higman-Thompson-descending-links:V}, the vertices of a descending link at height $m$ are the ordered $d$\=/tuples of distinct elements of a set of cardinality $n$, and two distinct vertices are compatible precisely when their supports are disjoint.
    In particular, the descending link has
    \[
        N_m\defq d!\binom{n}{d}
    \]
    vertices.
    A fixed vertex of the descending link is compatible with the $d!\binom{n-d}{d}$ vertices whose supports are disjoint from its support.
    It is compatible with itself as well.
    Thus, the descending link is uniformly $b_m$\=/dense, where
    \[
        b_m\defq d!\Bigl(\binom{n}{d}-\binom{n-d}{d}\Bigr)-1.
    \]
    For every $1\leq l\leq d$, summing Pascal's identity $l$ times gives
    \[
        \binom{n}{d}-\binom{n-l}{d}
        = \sum_{j=0}^{l-1} \Bigl(\binom{n-j}{d}-\binom{n-j-1}{d}\Bigr)
        = \sum_{j=0}^{l-1}\binom{n-j-1}{d-1}
        \leq l\binom{n-1}{d-1}.
    \]
    Taking $l=d$ and cancelling the factor $d!$ in the definitions of $b_m+1$ and $N_m$, we obtain
    \begin{equation}%
        \label{eq:thompson-V-a}
        \frac{b_m+1}{N_m} = \frac{\binom{n}{d}-\binom{n-d}{d}}{\binom{n}{d}} \leq \frac{d\binom{n-1}{d-1}}{\binom{n}{d}} = \frac{d^2}{n}.
    \end{equation}
    Assume now that $m>a_0$.
    Since $\heightrank{m-1}=n-d+1$, taking $l=d-1$ and again cancelling the factor $d!$ gives
    \begin{equation}%
        \label{eq:thompson-V-b}
        \frac{\abs{N_m-N_{m-1}}}{N_m} = \frac{\binom{n}{d}-\binom{n-d+1}{d}}{\binom{n}{d}} \leq \frac{(d-1)\binom{n-1}{d-1}}{\binom{n}{d}} =\frac{d(d-1)}{n}.
    \end{equation}
    Since $n=r+(d-1)m\geq(d-1)m$, combining~\cref{eq:thompson-V-a,eq:thompson-V-b} yields
    \[
        \frac{b_m+\abs{N_m-N_{m-1}}}{N_m}
        \leq \frac{d^2+d(d-1)}{n}
        = \frac{2d^2 - d}{n}
        \leq \frac{\altoverlapconst}{m}.\qedhere
    \]
\end{proof}

\begin{cor}[store=main-Higman-Thompson-fillings]%
    \label{cor:Higman-Thompson-polynomial-fillings}
    Let $d\geq2$ and $r\geq1$, and let $G$ be one of the groups
    $F_{d,r}$, $T_{d,r}$, or $V_{d,r}$.
    Then $\filling{G}{k}$ is polynomial for every $k\geq2$.
\end{cor}
\begin{proof}
    Fix $k\geq2$.
    By~\cref{prop:F-strong-uniform-density,prop:T-strong-uniform-density,prop:V-strong-uniform-density,cor:CAT0-ambient-fillings}, a sufficiently high sublevel set $\sublevel{X}{m}$ is $(k-1)$\=/acyclic and has polynomial integral filling functions in every degree from $2$ to $k$.

    All sublevel sets are cocompact and have finite cell stabilizers.
    Therefore,~\cite{young}*{Lemma~1} implies that $\filling{G}{k} \approx \filling{\sublevel{X}{m}}{k}$.
    Hence, $\filling{G}{k}$ is polynomial.
    Of course, this also implies that every $(k-1)$\=/acyclic sublevel set has polynomial filling functions.
\end{proof}

\section{Application to Röver--Nekrashevych groups}%
\label{sec:Rover-Nekrashevych-strong-uniform-density}

Let $d\geq2$ and let $G\leq\Aut(\rootedtree{d})$ be self-similar.
Recall that $V_d(G)$ denotes the \emph{Röver--Nekrashevych group} associated with $G$.
We refer to~\cite{SkipperZaremsky2024} for the definition, as we do not require it here.
Let $\RNcomplex{d}{G}$ be its Stein--Farley cube complex (see~\cite{SkipperZaremsky2024}*{Section~4.1}).
The vertices of $\RNcomplex{d}{G}$ carry a rank function $\vertexrank{\placeholder}$, whose values are congruent to $1$ modulo $d-1$, and we put
\[
    h(v)\defq\frac{\vertexrank{v}-1}{d-1}.
\]
A vertex of height $m$ has rank $\heightrank{m} \defq1+(d-1)m$.

The complex $\RNcomplex{d}{G}$ is contractible by~\cite{SkipperZaremsky2024}*{Proposition~4.2}.
It is shown in~\cite{Wu2025}*{Theorem~3.16} that $\RNcomplex{2}{G}$ is \CAT{0}.
The same argument extends almost verbatim to the case $d > 2$.

To verify strong uniform density and the large overlap condition, we use the following description of the descending links.
Recall that a simplicial map $\pi\colon K\to L$ is a \emph{complete join} if it is surjective, its restriction to every simplex is injective, and for every simplex $\{x_0,\ldots,x_q\}$ of $L$, the preimage subcomplex is the join
\[
    \pi^{-1}(x_0)*\cdots*\pi^{-1}(x_q).
\]
The construction in the proof of~\cite{LlosaIsenrichScheslerWu2025}*{Proposition~3.5} associates
to every vertex $v$ a finite set $B_v$ with $\size{B_v}=\vertexrank{v}$, together with a complete-join map
\[
    \pi_v\colon\lkdown(v)\longrightarrow
    \matchingcomplex{\completehypergraph*{d}{B_v}}.
\]
Note that $\matchingcomplex{\completehypergraph*{d}{B_v}}^{(0)}$
can naturally be identified with the set $\powerset[d]{B_v}$ of subsets of $B_v$ with cardinality $d$, so $\pi_v(a) \subseteq B_v$ for $a\in\lkdown(v)^{(0)}$.

These data satisfy the following compatibility condition.
If $a\in\lkdown(v)^{(0)}$ and $w=v/a$, there is an element $\ast_a\in B_w$ and a bijection
$\sigma_a\colon B_v\setminus\pi_v(a)\longrightarrow B_w\setminus\{\ast_a\}$.
For every $A\in\matchingcomplex{\completehypergraph*{d}{B_v}}^{(0)}$
with $A \cap \pi_v(a) = \emptyset$, there is a bijection
\[
    \tau_{v,a} \colon \pi_v^{-1}(A)\longrightarrow\pi_w^{-1}(\sigma_a(A)), \quad e\longmapsto e/a.
\]
\begin{rem}%
    \label{rem:relation-forest-diagrams}
    For readers familiar with forest diagrams, the preceding construction has a simple interpretation.
    The set $B_v$ is the set of feet of a diagram representing $v$.
    A vertex $a\in\lkdown(v)^{(0)}$ corresponds to a simple contraction:
    the $d$ feet in $\pi_v(a)$ are identified with the leaves of a $d$\=/ary caret.
    After the contraction, the root of this caret becomes the new foot $\ast_a$ of the diagram representing $w=v/a$.
    The bijection $\sigma_a$ identifies each unaffected foot of $v$ with the corresponding foot of $w$.
\end{rem}

We are now ready to prove strong uniform density and the large overlap condition for these complexes.
Note that, in general, the descending links in these complexes are not finite, so~\cref{prop:full-descending-link-density-criterion} does not apply.
We therefore check~\cref{def:strong-uniform-density-descending-links,def:large-overlap-descending-links} directly.

\begin{prop}%
    \label{prop:Rover-Nekrashevych-strong-uniform-density}
    The complex $\RNcomplex{d}{G}$ has strong uniform density of descending links and satisfies the large overlap condition.
\end{prop}
\begin{proof}
    Choose an integer $a_0\geq1$ such that $\heightrank{m}\geq2d$ for every $m\geq a_0$.
    We first show how to obtain finite subsets $\densityset{v} \subseteq \lkdown(v)^{(0)}$ satisfying strong uniform density.
    Let $v$ be a vertex of height $m \geq a_0$.
    Because $\pi_v\colon\lkdown(v)^{(0)}\longrightarrow
        \powerset[d]{B_v}$ is surjective,
    there is a section $s_v \colon \powerset[d]{B_v} \to \lkdown(v)^{(0)}$.
    We define $\densityset{v} \defq s_v(\powerset[d]{B_v})$.
    Thus, $\densityset{v}$ contains exactly one element from each fibre of $\pi_v$.

    Now consider $x\in\lkdown(v)^{(0)}$.
    Let $A \in \powerset[d]{B_v}$ be such that $\pi_v(x) \cap A = \emptyset$, that is,
    $\pi_v(x)$ and $A$ are connected by an edge in $\matchingcomplex{\completehypergraph*{d}{B_v}}$.
    Because $\pi_v$ is a complete join map, $x$ is compatible with every vertex in $\pi_v^{-1}(A)$.
    In particular, $x$ is compatible with $s_v(A) \in \densityset{v}$.
    There are $\binom{\heightrank{m}-d}{d}$ such sets $A \in \powerset[d]{B_v}$,
    so $\densityset{v}$ is $b_m$\=/dense in $\lkdown(v)$, where
    \[
        b_m\defq N_m -\binom{\heightrank{m}-d}{d},
        \qquad
        N_m\defq\binom{\heightrank{m}}{d}.
    \]
    The calculation in~\cref{eq:thompson-V-a}, with the common factor $d!$ omitted, gives
    \[
        \frac{b_m}{N_m}
        =\frac{\binom{\heightrank{m}}{d}-\binom{\heightrank{m}-d}{d}}{\binom{\heightrank{m}}{d}}
        \leq\frac{d^2}{\heightrank{m}}
        \longrightarrow0.
    \]
    Thus, any choice of sections $s_v$ yields strong uniform density with the above constants.

    Note that it is not clear that an arbitrary choice of sections produces sets $\densityset{v}$ satisfying the large overlap condition.
    However, a careful coherent choice does, as we shall now explain.
    To this end, we consider an auxiliary undirected graph $\Gamma$ whose vertices are the pairs $(v,A)$ with $v\in\superlevel{\RNcomplex{d}{G}}{a_0}^{(0)}$ and $A\in\powerset[d]{B_v}$.
    The edges in $\Gamma$ are of the form $\{(v,A), (v/a,\sigma_a(A))\}$
    for an element $a \in \lkdown(v)^{(0)}$ such that $A$ is disjoint from $\pi_v(a)$.

    For each connected component of $\Gamma$, choose one of its vertices $(v_0,A_0)$ and one lift $s_{v_0}(A_0) \in \pi_{v_0}^{-1}(A_0)$
    of $A_0$.
    Let $\gamma$ be a path from $(v_0,A_0)$ to some vertex $(v,A)$.
    Using the bijections $\tau_{v,a}$ or their inverses, the lift $s_{v_0}(A_0)$ can be
    transported to $v$.
    Denote the resulting lift of $A$ in $\pi_v^{-1}(A)$ by $s_v^{\gamma}(A_0)$.
    It turns out that $s_v^{\gamma}(A_0)$ does not depend on the specific path $\gamma$.
    To see this, first observe that
    if $\gamma$ is only a single edge coming from an element $a \in \lkdown(v_0)^{(0)}$, we have
    $s_v^{\gamma}(A_0) = \tau_{v_0,a}(s_{v_0}(A_0))$.
    By definition of $\tau_{v_0,a}$, the edges $s_{v_0}(A_0)$ and $s_v^\gamma(A_0)$ are opposite edges of a square,
    so $\dualhyperplane{s_{v_0}(A_0)} = \dualhyperplane{s_v^\gamma(A_0)}$.
    The same conclusion holds when the edge is traversed in the opposite direction.
    Inductively, one obtains that
    $\dualhyperplane{s_v^{\gamma}(A_0)} = \dualhyperplane{s_v^{\gamma'}(A_0)}$
    whenever $\gamma$ and $\gamma'$ are paths starting at $(v_0, A_0)$ with the same endpoint.
    By~\cref{lem:hyperplane-transport:injective}, this implies that indeed $s_v^{\gamma}(A_0) = s_v^{\gamma'}(A_0)$.

    For every vertex $(v,A)$ of $\Gamma$, denote by $(v_0,A_0)$ the chosen base vertex
    of its path component in $\Gamma$ and choose a path $\gamma_{(v,A)}$ from
    $(v_0,A_0)$ to $(v,A)$.
    For a vertex $v$ of height $m \geq a_0$, let $s_v$ be the section
    defined by $s_v(A) \defq s_v^{\gamma_{(v,A)}}(A_0)$ and define
    \begin{align*}
        \densityset{v} \defq s_v\bigl(\matchingcomplex{\completehypergraph*{d}{B_v}}^{(0)}\bigr) = \bigl\{ s_v^{\gamma_{(v,A)}}(A_0) \mid A \in \powerset[d]{B_v} \bigr\}.
    \end{align*}
    By the above discussion, this set is well-defined and does not depend on the choices of the paths $\gamma_{(v,A)}$.

    We now verify the large overlap condition for these sets.
    Let $v,w\in\superlevel{\RNcomplex{d}{G}}{a_0}^{(0)}$ with $w=v/a$, where $h(v)=m>a_0$.
    For every set $A\in\powerset[d]{B_v\setminus\pi_v(a)}$,
    we have
    \[
        s_w(\sigma_a(A))
        =\tau_{v,a}(s_v(A))
        =s_v(A)/a,
    \]
    so $\dualhyperplane{s_w(\sigma_a(A))} = \dualhyperplane{s_v(A)}$
    by~\cref{lem:hyperplane-transport:transport}.
    Consequently,
    \begin{align*}
        \size{\densityhyperplanes{v}\cap\densityhyperplanes{w}}
         &\geq
        \size{\bigl\{\dualhyperplane{s_v(A)}\bigm|
            A\in\powerset[d]{B_v\setminus\pi_v(a)}\bigr\}} \\
         &=\binom{\heightrank{m}-d}{d}
        =N_m-b_m.
    \end{align*}

    Since $\size{\densityhyperplanes{v}}=N_m$ and
    $\size{\densityhyperplanes{w}}=N_{m-1}\leq N_m$, we obtain
    \begin{align*}
        \size{\densityhyperplanes{v}\symdiff\densityhyperplanes{w}}
         &=N_m+N_{m-1}-2\size{\densityhyperplanes{v}\cap\densityhyperplanes{w}} \\
         &\leq N_{m-1}-N_m+2b_m \\
         &\leq2b_m.
    \end{align*}
    Since $\size{\densityhyperplanes{v}\cup\densityhyperplanes{w}}\geq N_m$ and
    $\heightrank{m}=1+(d-1)m\geq(d-1)m$, it follows that
    \begin{align*}
        \frac{\size{\densityhyperplanes{v}\symdiff\densityhyperplanes{w}}}
        {\size{\densityhyperplanes{v}\cup\densityhyperplanes{w}}}
         &\leq\frac{2b_m}{N_m}
        \leq\frac{2d^2}{\heightrank{m}}
        \leq\frac{2d^2}{(d-1)m}.
    \end{align*}
    This verifies the large overlap condition with $\overlapconstant\defq2d^2/(d-1)$.
\end{proof}

The $V_d(G)$\=/action on every sublevel set of $\RNcomplex{d}{G}$ is cocompact by~\cite{SkipperZaremsky2024}*{Proposition~4.2}.
If $G$ is finite, its cell stabilizers are finite, again by~\cite{SkipperZaremsky2024}*{Proposition~4.2}.
Consequently, the same argument as for~\cref{cor:Higman-Thompson-polynomial-fillings} applies.
Together with~\cref{prop:Rover-Nekrashevych-strong-uniform-density,cor:CAT0-ambient-fillings}, this gives the following result.

\begin{cor}[store=main-Rover-Nekrashevych-fillings]%
    \label{cor:Rover-Nekrashevych-polynomial-fillings}
    Let $d\geq2$ and let $G\leq\Aut(\rootedtree{d})$ be finite and self-similar.
    Then $\filling{V_d(G)}{k}$ is polynomial for every $k\geq2$.
\end{cor}

\bibliography{references}

\end{document}